\documentclass[10pt]{article}

\usepackage[english]{babel}
\usepackage[utf8]{inputenc}
\usepackage{amsmath,amsthm,amsfonts,amssymb}
\makeatletter
\g@addto@macro\normalsize{%
  \setlength\abovedisplayskip{4pt plus 2pt minus 1pt}%
  \setlength\belowdisplayskip{4pt plus 2pt minus 1pt}%
  \setlength\abovedisplayshortskip{2pt plus 2pt}%
  \setlength\belowdisplayshortskip{3pt plus 2pt minus 1pt}%
}
\makeatother

\usepackage[hmargin=2.5cm, tmargin=2.5cm, bmargin=2.5cm]{geometry}
\usepackage{graphicx}
\usepackage{subcaption}
\usepackage{hyperref}
\usepackage{cleveref}
\usepackage{enumerate}
\usepackage{enumitem}
\usepackage{empheq}
\usepackage{cite}
\usepackage{multirow}
\usepackage{multicol}
\usepackage{upgreek} 
\usepackage[bbgreekl]{mathbbol}
\usepackage{chemmacros}
\usepackage{xurl}
\usepackage{tabularx}
\usepackage{booktabs}
\usepackage[ruled,vlined]{algorithm2e}
\usepackage{makecell}
\usepackage[labelfont=bf,format=plain,justification=centering,singlelinecheck=false]{caption}
\usepackage{dsfont}
\usepackage{pifont} 
\usepackage{stmaryrd}
\usepackage{listings}
\usepackage{color} 
\usepackage{cases}
\usepackage{placeins}
\usepackage{transparent} 
\usepackage{eso-pic} 
\usepackage{tikz} 
\usetikzlibrary{}
\usepackage{subcaption}
\usepackage{xcolor} 
\usepackage{float}
\usepackage{mathrsfs}
\usepackage{todonotes}

\hypersetup{colorlinks, linkcolor=blue, citecolor=magenta,
urlcolor=magenta, plainpages=false, pdfwindowui=false,
pdfstartview={FitH}, pdftitle={A}}

\numberwithin{equation}{section}
\allowdisplaybreaks

\newtheorem{definition}{Definition}[section]
\newtheorem{lemma}{Lemma}[section]
\newtheorem{prop}{Proposition}[section]
\newtheorem{theorem}{Theorem}[section]

\newtheoremstyle{remarkstyle}{\topsep}
                             {\topsep}
			     {}            
			     {}            
			     {\scshape}    
			     {.}           
			     { }           
			     {\thmname{#1}\thmnumber{ #2}\thmnote{ #3}}
\theoremstyle{remarkstyle}
\newtheorem{remark}{Remark}[section]

\numberwithin{definition}{section}
\numberwithin{lemma}{section}
\numberwithin{prop}{section}
\numberwithin{theorem}{section}
\numberwithin{corol}{section}
\numberwithin{remark}{section}

\newcommand\widecheck[1]{%
    #1%
    \text{\llap{%
        \smash{\raisebox{\dimexpr 2\depth-0.4ex\relax}{%
            \scalebox{1}[-1]{$\widehat{\phantom{#1}}$}%
        }}%
    }}%
}

\newcommand{\fs}{\tau} 
\newcommand{\fv}{\boldsymbol{\tau}} 
\newcommand{\ov}{\boldsymbol{u}} 
\newcommand{\ovc}{u} 
\newcommand{\ovf}{U} 
\newcommand{\bv}{\boldsymbol{e}} 
\newcommand{\rhsv}{\boldsymbol{b}} 
\newcommand{\ovw}{\boldsymbol{v}} 
\newcommand{\fsx}{f} 
\newcommand{\fsy}{g} 
\newcommand{\fst}{h} 
\newcommand{\resns}{\overline{\delta}} 
\newcommand{\resnv}{\overline{\boldsymbol{\delta}}} 
\newcommand{\NLresns}{\delta} 
\newcommand{\NLresnv}{\boldsymbol{\delta}} 
\newcommand{\aef}{\psi} 
\newcommand{\nbsp}{\mathfrak{n}} 
\newcommand{\nbel}{\mathfrak{e}} 
\newcommand{\nbre}{\mathfrak{r}} 
\newcommand{\nbbad}{\mathfrak{m}} 

\newcommand{\nmole}{\textsf{n}}
\newcommand{\nmolevec}{\boldsymbol{\nmole}}
\newcommand{\NN}{\mathbb{N}}
\newcommand{\RR}{\mathbb{R}}
\newcommand{\RN}{\mathbb{R}^N}

\newcommand{\HHO}{\mathrm{H_2O}}
\newcommand{\Hp}{\mathrm{H}^+}
\newcommand{\OHm}{\mathrm{OH}^-}
\newcommand{\NaCl}{\mathrm{NaCl}}

\title{Acceleration of Newton's method for nonlinear algebraic systems by preflattening and elimination techniques}

\author{Ngoc Do Quyen \textsc{Dang}\thanks{IFP Energies nouvelles, 1 et 4 avenue de Bois Pr\'eau, 92852 Rueil-Malmaison Cedex, France. \texttt{ngoc-do-quyen.dang@ifpen.fr}, \texttt{ibtihel.ben-gharbia@ifpen.fr}, \texttt{quang-huy.tran@ifpen.fr}}
    \and Ibtihel \textsc{Ben Gharbia}${}^{*}$
    \and Clément \textsc{Canc\`{e}s}\thanks{Univ. Lille, CNRS, Inria, UMR 8524 -- Laboratoire Paul Painlev\'e, 59000 Lille, France. \texttt{clement.cances@inria.fr}}
    \and  Quang Huy \textsc{Tran}${}^{*}$}

\begin{document}

\maketitle

\begin{abstract}
In this paper, we propose a new approach called {\em preflattening} to speed up the resolution of nonlinear algebraic systems by Newton-type methods. Preflattening is meant to be for nonlinear systems what preconditioning is for linear ones. The idea is to transform the system into an equivalent one that is more favorable to Newton's method. To this end, we introduce the notion of {\em flatness}, which is intended to play for nonlinear systems a role similar to that of the condition number for linear systems. For a physical model of interest in the geosciences, namely the chemical equilibrium problem, we combine this strategy with the nonlinear elimination technique and demonstrate the value of the resulting approach through numerical simulations.
\end{abstract}
{\small
\begin{center}
\textbf{Keywords}\medskip\\
Newton's method, preconditioning, preflattening, chemical equilibrium
\end{center}
\begin{center}
\textbf{Mathematics subject classification}\medskip\\
65H10, 49M15, 65F35
\end{center}
}

\section{Introduction}
\subsection{Context}
The numerical discretization of physical models in many fields often leads to very large nonlinear algebraic systems. Solving such systems is a major challenge in industrial applications. Traditionally, Newton's method is employed for this purpose. Its main advantage lies in its high rate of convergence. Under appropriate assumptions on the initial point and on the residual function, it achieves superlinear or even quadratic convergence.

In practice, however, Newton's method suffers from a few well-known drawbacks. First, each Newton iteration requires the solution of a linear system, an operation that becomes expensive for high-dimensional problems. Second, it is very difficult to select an initial guess that ensures convergence by satisfying the assumptions of the available theoretical results \cite{Kantorovich1948,OrtegaRheinboldt1970,Smale1986,Deuflhard2011}. The lack of a judicious starting point may cause slow convergence or divergence. Third, strong nonlinearities in the equations, such as those induced by large time steps, bad choices of variables or stiff closure laws, are likely to deteriorate the performance of the method. Finally, unbalanced nonlinearities are also a source of trouble. Hence, the quest for an enhanced {\em effectiveness} of Newton's method remains an extensive area of research. This is the context in which our work is set.

Over the past two decades, a class of techniques known as {\em nonlinear preconditioning} has proved to be a great help in accelerating the resolution of nonlinear algebraic systems. The basic idea of nonlinear preconditioning is to replace the original system by an equivalent\footnote{in the sense that the new system has the same solution, up to an easily computable change of variables.} one, which lends itself more favorably to Newton's method or to its variants.

The construction of nonlinear preconditioners proceeds by analogy with that of linear ones: most of the mathematical machinery developed for the linear case can be formally mimicked or straightforwardly carried over to the nonlinear case. Without claiming exhaustiveness, let us highlight some notable developments, including additive methods such as ASPIN \cite{CaiKeyes2002} and RASPEN \cite{DoleanGKKM2016}, multiplicative methods such as  INB-NE \cite{HwangLinCai2010}, MSPIN \cite{LiuKeyes2015} and NEPIN \cite{LiuHLCK2022}, nonlinear versions of FETI-DP \cite{KlawonnLanserRheinbach2014} and BDDC \cite{KlawonnLanserRheinbach2018}. It is also worth mentioning a simpler family of methods known as Jacobi-Newton \cite{Brenner2023, Brenner2020}, dedicated to systems with ``mild'' nonlinearities.

Three recent sets of work addressed various aspects of nonlinear systems through techniques related to nonlinear preconditioning. Bassetto and her coauthors \cite{Bassetto2021,BassettoCET2020} extended the Brenner--Canc\`es parametrization \cite{BrennerCances2017} to realistic relative permeability laws and capillary pressure-saturation relationships in porous media flow. This may be interpreted as a nonlinear right-preconditioner. Jonval and his coauthors \cite{Jonval2024,JonvalBGCFT2025,Jonval2025b} not only applied the Brenner--Canc\`es parametrization to ideal activity models in equilibrium chemistry, but also put forward a so-called Cartesian representation technique. The latter can be viewed as a nonlinear left-preconditioner. Marelli and her coauthors \cite{Marelli2025,BrennerMarelliTran2025,BrennerMarelliTran2026} explored dynamic versions of nonlinear preconditioners based on domain decomposition and variable splitting, in order to accelerate the solution of nonlinear systems arising from PDEs. Emphasis is placed on nonlinear elimination \cite{LanzkronRoseWilkes1996} (see also method D2 of \cite{BassettoCET2022}). Nonlinear preconditioning can thus serve multiple goals: to enforce robustness by preventing the blow-up of derivatives in all regimes \cite{Bassetto2021,Jonval2024} and to increase speed by reducing the number of iterations \cite{Marelli2025}.

\subsection{Objectives}
There is an important caveat to the design of nonlinear preconditioners by formal analogy with linear ones. In the linear world, the {\em condition number} of the preconditioned system can be made lower than that of the original system. Due to the remarkable coincidence that the condition number also governs the rate of convergence of iterative linear solvers, speed-up is guaranteed. In the nonlinear setting, by contrast, no analogous scalar quantity can be identified that simultaneously: (i) decreases between the unpreconditioned and preconditioned systems; (ii) governs the rate of convergence in terms of Newton iterations.

Actually, a notion of condition number can be defined \cite{Rheinboldt1976,Rice1966} for a nonlinear function $F$. This quantity depends on the current point $\ov$ and a surrounding neighborhood. It involves first-order properties of $F$ and its inverse $F^{-1}$ and only serves as a measure of the sensitivity of the relative error in the solution of $F(\ov)=\rhsv$ with respect to perturbations in $\rhsv$ and $F$. Such a notion of conditioning never appears in convergence analyses of Newton's method, where second-order information is required. Even assuming that it plays a meaningful role in the convergence rate of Newton, the current preconditioning strategies do not aim to diminish it.

These observations point to a relative weakness in the foundation of nonlinear preconditioning and prompt us to develop a less heuristic approach, which we refer to as {\em preflattening}. The basic tenet is to transform the given system into a locally ``less nonlinear'' one. Indeed, it is well known that Newton's method converges in a single iteration for an affine function. Intuitively, the closer a function is to being affine, the better Newton's method will behave. The advantage of preflattening is to act exclusively at the nonlinear level of Newton's method, thus disentangling the two sources of difficulty and leaving the linear level to existing linear preconditioners.

To implement this philosophy, we introduce the notion of {\em flatness number} $\rho_F(\ov)$ in \S\ref{ch:notion}. The latter, which appears as a dimensionless measure of the nonlinearity of $F$ at $\ov$, is furthermore connected, albeit not identical, to the rate of convergence of Newton such as provided by classical theorems. {\em Right-preflattening} then amounts to solving $F(\ovf(\fv))=0$ for $\fv$, followed by $\ov=\ovf(\fv)$, instead of solving $F(\ov)=0$ for $\ov$. The change of variables $U$ is performed at each iteration so as to ensure $\rho_{F\circ \ovf}(\fv) < \rho_F(\ov)$ at the same iteration.

After some prerequisites in \S\ref{ch:notion}, the explicit construction of a system-dependent $U$ is given in \S\ref{ch:scalar} and \S\ref{ch:chemistry}, where we advocate the {\em Continuous Dynamic} (CD) method by minimizing $\rho_{F\circ U}(\ovf^{-1}(\ov))$ over a parametrized family of candidates $\ovf$ at the current iterate $\ov$. While the CD method works extremely well for the scalar equation of \S\ref{ch:scalar}, it does not yield satisfactory results for systems. This is why we envisage a combination of CD with the {\em nonlinear elimination} strategy in \S\ref{ch:chemistry}, which allows us to recover substantial gains in the number of iterations.

\section{Notion of flatness and preflattening}\label{ch:notion}
\subsection{Notion of flatness}\label{sse:notion}
We consider an algebraic system of equations of the form 
\begin{equation}\label{eq:orisystem}
    F(\ov)=0, 
\end{equation}
where $F:\Omega\subseteq \RN \rightarrow \RN$ is a vector-valued function. For simplicity, let us assume that: (i) $F$ is twice continuously differentiable on $\Omega$; (ii) the solution $\ov^*$ of $\eqref{eq:orisystem}$ is unique; (iii) the Jacobian matrix $\nabla F(\ov)$ is always nonsingular. Given an initial guess $\ov^{0}\in \Omega$, we wish to solve \eqref{eq:orisystem} by means of the Newton method
\begin{equation}\label{eq:orinewton}
    \ov^{k+1} = \ov^{k} - \nabla F(\ov^k)^{-1} F(\ov^k),
\end{equation}
in which the superscript $k\in\NN$ denotes the iteration number.

We are interested in defining a quantity that measures the local nonlinearity of $F$ and that also reflects the local difficulty of Newton's method. To this end, we see \eqref{eq:orinewton} as a fixed-point iteration
\begin{equation}
    \ov^{k+1} = \Phi_F(\ov^k),
\end{equation}
where
\begin{equation}\label{eq:iterfunc}
    \Phi_F(\ov)=\ov- \nabla F(\ov)^{-1} F(\ov)
\end{equation}
is known as the {\em iteration function}. According to Ostrowski's theorem \cite[\S10.1.3]{OrtegaRheinboldt1970} for general one-step iterative methods (not necessarily deriving from a Newton process), the spectral radius of the Jacobian matrix $\nabla\Phi_F$ at the solution $\ov^*$ plays a crucial role for convergence. This motivates us to consider this spectral radius for an arbitrary point $\ov$.

\begin{definition}\label{de:flatMATnum}
The Jacobian matrix of the iteration function
\begin{equation}\label{eq:flatMAT}
    \mathcal{J}_F(\ov) : = \nabla\Phi_F(\ov) = \nabla F(\ov)^{-1} \nabla^2 F(\ov) \nabla F(\ov)^{-1} F(\ov)
\end{equation}
is said to be the {\em flatness matrix} of $F$ at $\ov\in\Omega$. Its spectral radius
\begin{equation}\label{eq:flatnum}
    \rho_F(\ov) := \rho(\mathcal{J}_F(\ov))
\end{equation}
is said to be the {\em flatness number} of $F$ at $\ov\in\Omega$.
\end{definition}
The terminology should be understood as a measure of the departure from flatness. The smaller $\rho_F(\ov)$, the flatter the system at $\ov$, that is, the closer to affine, and $\rho_F\equiv 0$ when $F$ is affine. Note also that $\mathcal{J}_F(\ov^*)=0$ for every $F$, since $F(\ov^*)=0$, which is the classical reason for the superlinear convergence of Newton's method. The flatness number is therefore informative away from the solution.
The second expression in \eqref{eq:flatMAT} follows from differentiating \eqref{eq:iterfunc}, the identity terms cancelling out. For a scalar function $f : D\subseteq \RR\rightarrow \RR$, it readily gives
\begin{equation}\label{eq:flat.scalar}
    \mathcal{J}_f(u) = \frac{f'' (u)f(u)}{f'(u)^2}, \qquad
    \rho_f(u) = \bigg| \frac{f'' (u)f(u)}{f'(u)^2} \bigg|,
\end{equation}
and the value of $\mathcal{J}_f(u)$ coincides with the degree of {\em logarithmic convexity} \cite{Hernandez1991}. For a vector-valued function $F$, the factor $\nabla^2 F(\ov)$ in \eqref{eq:flatMAT} is a third-order tensor, which acts on the vector $\nabla F(\ov)^{-1} F(\ov)$ by contraction and does not have an easy representation. Despite this difficulty, the flatness matrix can be written explicitly column by column.

\begin{prop}\label{pr:flatMATexp}
For $j\in \{1,\ldots, N\}$, the $j$-th column of $\mathcal{J}_F(\ov)$ is given by
\begin{equation}\label{eq:flatMATjth}
    \nabla F(\ov)^{-1} \partial_j [\nabla F(\ov) ] \nabla F(\ov)^{-1} F(\ov),
\end{equation}
where matrix $\partial_j [\nabla F(\ov) ]$ is obtained by differentiating all entries of $\nabla F(\ov)$ with respect to $\ovc_j$.
\end{prop}
\begin{proof}
See \S\ref{sse:flatMATexp}.
\end{proof}

We prefer $\rho_F(\ov)$ to $\|\mathcal{J}_F(\ov)\|$ as an indicator of flatness because $\rho_F(\ov)$ does not depend on the choice of a norm and is, moreover, {\em affine-invariant} in the following sense. For a nonsingular matrix $A$ and a vector $\rhsv$, let us define the two functions
\begin{equation}
    {}_A F(\ov) = A F(\ov), \qquad F_{A,\rhsv} (\ovw) = F(A\ovw+\rhsv),
\end{equation}
which result from a linear combination of the equations of $F$ and from an affine transformation of its argument, respectively. Newton's method is renowned to be affine-invariant: (i) the iterates produced by ${}_A F$ are the same as those produced by $F$; (ii) the iterates of $F_{A,\rhsv}$ match those of $F$ through $\ov = A\ovw + \rhsv$. As a measure of nonlinearity, $\rho_F(\ov)$ is expected to enjoy the same property.

\begin{prop}\label{pr:invariance}
With the above notations, we have
\begin{subequations}\label{eq:invariance}
\begin{alignat}{2}
\mathcal{J}_{{}_A F}(\ov) &= \mathcal{J}_{F}(\ov) & \qquad \rho_{{}_A F}(\ov) &= \rho_{F}(\ov),\label{eq:invarianceLIN}\\
\mathcal{J}_{F_{A,\rhsv}}(\ovw) &= A^{-1}\mathcal{J}_{F}(\ov) A & \qquad \rho_{F_{A,\rhsv}}(\ovw) &= \rho_{F}(\ov).\label{eq:invarianceAFF}
\end{alignat}
\end{subequations}
\end{prop}
\begin{proof}
    See \S\ref{sse:invarianceJrho}.
\end{proof}

A spectral radius may not be easy to handle, and it is in general more convenient to work with the upper bound below. Here and throughout, a matrix norm is understood to be submultiplicative. In \S\ref{ch:chemistry}, we shall use \eqref{eq:comprhonormJ} with the Frobenius norm.
\begin{prop}\label{pr:comprhonormJ}
For all matrix norm $\|\cdot\|$ and all invertible matrix $P$ of dimension $N\times N$,
\begin{equation}\label{eq:comprhonormJ}
    \rho_F(\ov) \le \|P^{-1}\mathcal{J}_F(\ov) P\|.
\end{equation}
\end{prop}
\begin{proof}
See \S\ref{sse:flatnumupper}.
\end{proof}

For the iteration function $\Phi_F$ to be a contraction in a convex neighborhood of $\ov^*$, it is sufficient that $\|\mathcal{J}_F(\ov)\| \le \gamma < 1$ uniformly in this neighborhood, for some matrix norm $\|\cdot\|$ induced by a vector norm. By Proposition~\ref{pr:comprhonormJ}, this condition implies $\rho_F(\ov) \le \gamma < 1$ in the same neighborhood.

\subsection{Preflattening}\label{sse:def.preflat}
Modifying a system in order to make it faster to solve is commonly referred to as {\em preconditioning}. As with their linear counterparts, nonlinear preconditioners fall into two kinds, depending on how the original system \eqref{eq:orisystem} is modified. We speak of {\em left-preconditioning} if the modified system keeps the same unknowns as the original one, which generalizes the linear case $M^{-1}(A\ov-\rhsv)=0$. A {\em right-preconditioning} technique, on the other hand, is based on a change of variables $\ov = \ovf(\fv)$, which turns the original system into $F\circ\ovf(\fv)=0$ while preserving the residuals of the equations. This generalizes the linear case $F(\ov) = A\ov - \rhsv$ with $\ovf(\fv) = M^{-1}\fv$.

In the linear case, the condition number $\kappa$ governs the convergence bounds of iterative solvers \cite{Saad2003}. Therefore, a linear preconditioner is designed so that $\kappa(M^{-1}A)\ll \kappa(A)$ for left-preconditioning or $\kappa(AM^{-1})\ll \kappa(A)$ for right-preconditioning. In the nonlinear case, the existing preconditioners are not designed to decrease a clearly identified quantitative target, despite their proven effectiveness. In light of the elements presented above, we propose to use the flatness number $\rho_F$, or an upper bound $\widehat{\rho}_F$ of it, as such a target, and we coin the term {\em preflattening} for any process that decreases this number. Preflattening is meant to make a system less nonlinear, and it should leave an affine system unchanged.

Below is the prototype of the right-preflattened Newton step studied in this work. Here, $\{\ovf_{\boldsymbol{\alpha}}\}_{\boldsymbol{\alpha}\in \mathcal{A}}$ designates a family of invertible parametrizations $\ov = \ovf_{\boldsymbol{\alpha}}(\fv)$, with nonsingular Jacobian matrices, that includes the identity and for which we are able to carry out the minimization \eqref{eq:minflat}. The index set $\mathcal{A}$ can be discrete or continuous, and the choice of the family strongly depends on the model under consideration. Given the current iterate $\ov^k$, we successively compute
\begin{subequations}\label{eq:abstractalgo}
\begin{alignat}{2}
\boldsymbol{\alpha}^k & = \text{arg} \textstyle\min_{\boldsymbol{\alpha}\in \mathcal{A}} \big\{ \rho_{F\circ \ovf_{\boldsymbol{\alpha}}}(\ovf_{\boldsymbol{\alpha}}^{-1}(\ov^k)) \big\}, & \ & \text{[optimal preflattening]}\label{eq:minflat}\\
\fv^k & = \ovf_{\boldsymbol{\alpha}^k}^{-1}(\ov^k), & \ & \text{[current variable]}\label{eq:realtau}\\
\fv^{k+1} & = \fv^k - [\nabla_{\fv} (F\circ \ovf_{\boldsymbol{\alpha}^k})(\fv^k)]^{-1} (F\circ \ovf_{\boldsymbol{\alpha}^k})(\fv^k)\nonumber\\
& = \fv^k - \nabla_{\fv} \ovf_{\boldsymbol{\alpha}^k}(\fv^k)^{-1} \nabla_{\ov} F(\ov^k)^{-1} F(\ov^k),\label{eq:updatetau} & \ & \text{[updated variable]}\\
\ov^{k+1} & = \ovf_{\boldsymbol{\alpha}^k}(\fv^{k+1}), & \ & \text{[primary unknown]}\label{eq:backtou}
\end{alignat}
\end{subequations}
to get the next iterate. The second line of \eqref{eq:updatetau} follows from the chain rule, since $\ovf_{\boldsymbol{\alpha}^k}(\fv^k)=\ov^k$. Note that $\ov$ is the variable of interest, while $\fv$ only works behind the scenes. Its updated value \eqref{eq:updatetau} is only used in \eqref{eq:backtou}, and it is overwritten by \eqref{eq:realtau} at the next iteration.

\section{Preflattening a scalar equation}\label{ch:scalar}
In this section, we apply the right-preflattening strategy \eqref{eq:abstractalgo} to a scalar model problem and compare it with alternative methods.
\subsection{Model}\label{sse:sca.model}
For parameters\footnote{Strict positivity of both $p$ and $q$ is needed for the switching point \eqref{scalar.switch} and for the conversions \eqref{BC.alg.scalar} to be well defined. The restriction $m>2$ guarantees that the second derivative $\fsy''(y)=pm(m-1)\{y\}^{m-2}$ of the function $\fsy$ defined in \eqref{eq:scalar.y} is continuous at $y=0$, hence that $\fsy$ is of class $C^2$ as required in \S\ref{sse:notion}. By contrast, the function $\fsx$ defined in \eqref{eq:scalar.x} is not even differentiable at $x=0$.} $(p,q) \in (\mathbb{R}_+^*)^2$, $r\in \mathbb{R}_+$ and $m>2$, consider the scalar equation
\begin{equation}\label{eq:scalar.x}
    \fsx(x):=px+q\{x\}^{1/m}-r=0,
\end{equation}
where, for any real exponent $a$, we use the signed power $\{x\}^{a}=\operatorname{sgn}(x)|x|^{a}$, with the convention $\{0\}^{a}=0$ for $a>0$. Here and below, $\mathbb{R}_+$ and $\mathbb{R}_+^*$ denote the sets of nonnegative and of positive real numbers, respectively. As $x\rightarrow 0$, the derivative $\fsx'(x)$ blows up, which can cause Newton's method to fail.

A standard remedy is to change the variable to $y=\{x\}^{1/m}$, which gives the equivalent equation
\begin{equation}\label{eq:scalar.y}
    \fsy(y):=p\{y\}^m+qy-r=0.
\end{equation}
Now, however, $\fsy'(y)$ blows up as $y\rightarrow\pm\infty$. Neither formulation is thus robust on its own, and a good method must switch between $\fsx$ and $\fsy$ depending on the regime. We therefore look for a unified formulation
\begin{equation}\label{scalar.eq}
    \fst(\fs):=pX(\fs)+qY(\fs)-r=0,
\end{equation}
where $\fs\mapsto X(\fs)$ is an increasing bijection of $\RR$ and $Y(\fs)= \{X(\fs)\}^{1/m}$, or equivalently $X(\fs)=\{Y(\fs)\}^m$. Since $\fst(\fs) = \fsx(X(\fs)) = \fsy(Y(\fs))$, problem \eqref{scalar.eq} is a general right-preconditioned version of \eqref{eq:scalar.x} or \eqref{eq:scalar.y}, and the choice of the parametrization determines the method.

\subsection{Methods}\label{sse:sca.methods}
\paragraph*{Binary Switch (BS).} In the simplest case, the family of parametrizations is indexed by $\mathcal{A}=\{0,1\}$ and consists of the two extreme choices $(X_0(\fs), Y_0(\fs)) = (\fs, \{\fs\}^{1/m})$ and $(X_1(\fs), Y_1(\fs)) = (\{\fs\}^m, \fs)$, which correspond to \eqref{eq:scalar.x} and \eqref{eq:scalar.y} respectively. Since $\fsx$ is not differentiable at $x=0$, and since both flatness numbers vanish at a solution of \eqref{eq:scalar.x}, the comparison below is understood at points $x\ne 0$ such that $\fsx(x)\ne 0$, which are the only ones of interest for choosing the next Newton step.
\begin{prop}\label{pr:rhosmaller}
    Let $\fsx$ and $\fsy$ be defined as in \eqref{eq:scalar.x} and \eqref{eq:scalar.y}. Then, for all pairs $(x,y)\in\mathbb{R}^2$ such that $y= \{x\}^{1/m}$, we have
    \begin{equation}\label{rho.scalar}
    \rho_{\fsy}(y)<\rho_{\fsx}(x) \;\; \Longleftrightarrow \;\; |y|<\left(\dfrac{q}{pm}\right)^{1/(m-1)}.
\end{equation}
\end{prop}
\begin{proof}
    See \S\ref{proof:rho.smaller}.
\end{proof}

For convenience, let us set
\begin{equation}\label{scalar.switch}
    (x_\flat,y_\flat)=\bigg(\!\bigg(\dfrac{q}{pm}\bigg)^{\!m/(m-1)}, \; \bigg(\dfrac{q}{pm}\bigg)^{\!1/(m-1)}\, \bigg)
\end{equation}
and call $(x_\flat,y_\flat)$ the {\em switching point}. Owing to \eqref{rho.scalar}, the right-preflattening framework \eqref{eq:abstractalgo} can be reformulated as follows. Given the current iterate $(x^k, y^k)$ satisfying $y^k = \{x^k\}^{1/m}$,
\begin{itemize}
    \item if $|y^k|< y_\flat$, run Newton's method on $\fsy$ to get $y^{k+1}$
    \begin{subequations}\label{eq:binaryswitch}
    \begin{equation}
        y^{k+1} = y^k - \fsy(y^k)/\fsy'(y^k),\qquad
        x^{k+1} = \{y^{k+1}\}^{m};\;
    \end{equation}
    \item if $|y^k| \ge y_\flat$, run Newton's method on $\fsx$ to get $x^{k+1}$
    \begin{equation}
      \;\;  x^{k+1} = x^k - \fsx(x^k)/\fsx'(x^k),\qquad
        y^{k+1} = \{x^{k+1}\}^{1/m}.
    \end{equation}
    \end{subequations}
\end{itemize}
We refer to algorithm \eqref{eq:binaryswitch} as {\em Binary Switch} (BS). The name underscores the dynamic switching between the two models, since at each iteration we select, among $\fsx$ and $\fsy$, the one that has the smaller flatness number at the current iterate.

\paragraph*{Brenner--Canc\`es (BC).} BS is not a pure Newton method, since the function being linearized changes from one iteration to the next, so that the usual convergence theory does not directly apply. To fix this while still exploiting \eqref{rho.scalar}, we move the switch inside the parametrization and set
\begin{subequations}
\begin{equation}
(X(\fs), Y(\fs)) =\begin{cases}
         (\{\fs\}^{m},\fs) \qquad &\text{\rm if }\, |\fs|< y_\flat,\\
         (Z(\fs), \{Z(\fs)\}^{1/m}) \qquad  &\text{\rm if } \, |\fs| \ge y_\flat,
    \end{cases} \label{BC.scalar}
\end{equation}
where
\begin{equation}
    Z(\fs) = \dfrac{q}{p}(\fs-\operatorname{sgn}(\fs)y_\flat)+\operatorname{sgn}(\fs)x_\flat.
\end{equation}
\end{subequations}
This is the Brenner--Canc\`es parametrization \cite{BrennerCances2017}, up to the choice of the switching point. The original method uses
\begin{equation}\label{eq:oldswitchpoint}
    (\widetilde{x}_\flat,\widetilde{y}_\flat) = \bigg(\!\bigg(\dfrac{1}{m}\bigg)^{\!m/(m-1)},\, \bigg(\dfrac{1}{m}\bigg)^{\!1/(m-1)}\;\bigg),
\end{equation}
for which $X'(\widetilde{y}_\flat) = 1$.

Intuitively, $\fs$ plays the role of $y$ for $|\fs|< y_\flat$ and, up to an affine map, the role of $x$ for $|\fs|\ge y_\flat$, so that BC merges the two formulations into a single variable. The functions $X$ and $Y$ are continuously differentiable (though not twice) at the switching point, which is all that Newton's method requires. BC is therefore a genuine Newton method. Starting from $(x^0,y^0)$ with $y^0 = \{x^0\}^{1/m}$,
\begin{itemize}
    \item compute the initial value
    \begin{subequations}\label{BC.alg.scalar}
    \begin{equation}
        \fs^0=\begin{cases}
            y^0,\quad &\text{\rm if }\; |y^0 | < y_\flat,\\
            \frac{p}{q} (x^0-\operatorname{sgn} (x^0) x_\flat )+\operatorname{sgn}(x^0) y_\flat,\quad &\text{\rm if }\; |y^0| \ge y_\flat;
        \end{cases}
    \end{equation}
    \item run Newton's method on $\fst$ until the stopping criterion \eqref{eq:stoppingcriterion} is met
    \begin{equation}
        \fs^{k+1} = \fs^k - \fst(\fs^k) / \fst'(\fs^k), \qquad x^{k+1} = X(\fs^{k+1}), \qquad y^{k+1} = Y(\fs^{k+1}).
    \end{equation}
    \end{subequations}
\end{itemize}

\paragraph*{Continuous Dynamic (CD).} BS uses only two parametrizations. CD instead relies on a continuous family indexed by $\alpha\in [0,1]$,
\begin{equation}\label{CD.para}
    (X_{\alpha}(\fs), Y_{\alpha}(\fs)) = \Big( \{\fs\}^{\frac{m}{(1-\alpha) m+\alpha}}, \, \{\fs\}^{\frac{1}{(1-\alpha) m+\alpha}} \Big).
\end{equation}
For $\alpha=0$, formula \eqref{CD.para} reduces to $(X_0(\fs), Y_0(\fs)) = ( \fs,\,\{\fs\}^{1/m} )$, which amounts to taking $x$ as the unknown. For $\alpha=1$, it becomes $(X_1(\fs), Y_1(\fs)) = (\{\fs\}^m,\,\fs)$, which amounts to taking $y$ as the unknown. The parametrization \eqref{CD.para} thus continuously interpolates between the two extreme choices of BS, which ensures a smooth transition. The idea is to choose $\alpha$ at each iterate so as to minimize, and in fact cancel, the flatness number of
\begin{equation}\label{eq:halphatau}
    \fst_\alpha(\fs) = pX_\alpha(\fs) + q Y_\alpha(\fs)-r.
\end{equation}
In the following statement, $\fs\ne 0$, since the right-hand side of \eqref{alpha.1D} is undefined at $x=y=0$ (see Remark~\ref{re:alphaCDzero}).

\begin{prop}\label{pr:alphaCDscalar}
    Let $\fst_\alpha$ be the function defined in \eqref{eq:halphatau}. Then, for all $(\alpha,\fs, x, y)\in [0,1]\times \mathbb{R}^3$ such that $(x,y)=(X_\alpha(\fs), Y_\alpha(\fs))$,
    \begin{equation}\label{alpha.1D}
    \fst^{\prime\prime}_\alpha(\fs) = 0 \;\;\Longleftrightarrow \;\;
    \alpha=\frac{qy}{pmx+qy}.
    \end{equation}
\end{prop}
\begin{proof}
    See \S\ref{proof:opt.para}.
\end{proof}
\begin{remark}\label{re:alphaCDzero}
The right-hand side of \eqref{alpha.1D} takes the indeterminate form $0/0$ at $x=y=0$. Since $x=\{y\}^m$ along the parametrization, it can be simplified into
\begin{equation}\label{eq:alphaCDsimplified}
    \frac{qy}{pmx+qy} = \frac{q}{pm|y|^{m-1}+q},
\end{equation}
an expression that is free of any indeterminacy and defines a continuous map $y\mapsto\alpha\in(0,1]$ on the whole of $\mathbb{R}$, with $\alpha=1$ if and only if $y=0$. The value $\alpha=1$ selects $(X_1(\fs),Y_1(\fs)) = (\{\fs\}^m,\fs)$ in \eqref{CD.para}, that is, Newton's method on $\fsy$. This is indeed the sound choice at $y=0$, where $\fsx'$ blows up while $\fsy'(0)=q>0$, and it is also the choice made by BS, since $0<y_\flat$. Formula \eqref{eq:alphaCDsimplified} is the one used in \eqref{CD.alg.scalar} and in the numerical experiments. Besides removing the indeterminacy, it avoids forming the sum $pmx^k+qy^k$ of two vanishingly small quantities.
\end{remark}

Owing to this result, CD can be written as follows. Given the current iterate $(x^k, y^k)$ with $y^k= \{x^k\}^{1/m}$, compute successively
\begin{subequations}\label{CD.alg.scalar}
\begin{alignat}{2}
    \alpha^k &= q/ (pm|y^k|^{m-1} + q), & \qquad &\text{[tuning weight]}\\
    \fs^k &= \{y^k\}^{(1-\alpha^k)m+\alpha^k}, & \qquad &\text{[current variable]}\\
    \fs^{k+1} & = \fs^k - \fst_{\alpha^k}(\fs^k) / \fst'_{\alpha^k}(\fs^k), & \qquad &\text{[updated variable]}\\
    x^{k+1} & = X_{\alpha^k}(\fs^{k+1}),& \qquad &\text{[primary unknown]}\\
    y^{k+1} & = Y_{\alpha^k}(\fs^{k+1}). & \qquad &\text{[primary unknown]}
\end{alignat}
\end{subequations}

\subsection{Numerical results}\label{sse:sca.results}
We fix $m = 3$ and test the four configurations of $(p,q,r)$ listed in Table~\ref{tab:cases}. In Cases 1 and 2, $p=q=1$ and only $r$ varies. In Cases 3 and 4, $p$ is taken very small or very large, and $r$ is chosen so that the solution is an integer.
\begin{table}[h]
    \centering
    \begin{tabular}{|c|c|c|c|}
     \hline
     Case & $p$ & $\;\;q\;\;$ & $r$\\
     \hline
     1 & 1 & 1 & 2\\
     \hline
     2 & 1 & 1 & 0.19\\
     \hline
     3 & 0.01 & 1 & 1.01 \\
     \hline
     4 & 1000 & 1 & 8002\\
     \hline
    \end{tabular}
    \caption{Values of $(p, q, r)$ for each test case}
    \label{tab:cases}
\end{table}

\paragraph*{Standard stopping criterion.}
We compare Newton's method on $\fsx$, Newton's method on $\fsy$, BS, BC and CD for a range of starting points, using the stopping criterion
\begin{equation}\label{eq:stoppingcriterion}
    |\fst(\fs)| < 10^{-10},
\end{equation}
with a maximum of $\mathtt{max}\_\mathtt{iter} = 100$ iterations. Since $\fst(\fs) = \fsx(X(\fs)) = \fsy(Y(\fs))$, the residual in \eqref{eq:stoppingcriterion} is the same quantity $|\fsx(x)|$ for all methods, which makes the comparison fair. In Tables~\ref{tab:pow_1}--\ref{tab:pow_4}, ``\texttt{x}'' marks non-convergence.

\begin{table}
\begin{minipage}[b]{0.45\textwidth}
\centering
\begin{tabular}{|c|c|c|c|c|c|}
\hline
$\;x^0\;$ &  $\;\fsx\;$ &  $\;\fsy\;$  & BS & BC & CD \\
\hline 
3 & 4 & 5 & 4 & 4 & 3 \\
\hline 
9 & 4 & 6 & 4 & 4 & 3 \\
\hline 
20 & 5 & 6 & 5 & 5 & 3 \\
\hline 
30 & 5 & 7 & 5 & 5 & 3 \\
\hline 
100 & 6 & 8 & 6 & 5 & 4 \\
\hline 
1000 & 6 & 9 & 6 & 6 & 4 \\
\hline
\end{tabular}
\caption{Number of iterations in Case 1}
\label{tab:pow_1}
\end{minipage}\hfill
\begin{minipage}[b]{0.45\textwidth}
\centering 
\begin{tabular}{|c|c|c|c|c|c|}
\hline$\quad x^0\quad$ &  $\;\fsx\;$ &  $\;\fsy\;$  & BS & BC & CD \\
\hline 
$-100$ & \texttt{x} & 8 & 10 & 6 & 6 \\
\hline 
$-10$ & \texttt{x} & 7 & 10 & 5 & 6 \\
\hline 
$-1$ & \texttt{x} & 5 & 8 & 4 & 5 \\
\hline 
1 & \texttt{x} & 5 & 7 & 4 & 5 \\
\hline 
100 & \texttt{x} & 9 & 9 & 6 & 6 \\
\hline 
1000 & \texttt{x} & 11 & 11 & 6 & 7 \\
\hline
\end{tabular}
\caption{Number of iterations in Case 2}
\label{tab:pow_2}
\end{minipage}\hfill
\begin{minipage}[b]{0.45\textwidth}
\centering
\begin{tabular}{|c|c|c|c|c|c|}
\hline
$x^0$ & $\;\fsx\;$ & $\;\fsy\;$  & BS & BC & CD \\
\hline 
$-100$ & \texttt{x} & 4 & 4 & 4 & 4 \\
\hline 
$-10$ & \texttt{x} & 4 & 4 & 4 & 3 \\
\hline 
$-1$ & 13 & 3 & 3 & 3 & 3 \\
\hline 
10 & \texttt{x} & 3 & 3 & 3 & 3 \\
\hline 
100 & \texttt{x} & 4 & 4 & 4 & 4 \\
\hline 
1000 & \texttt{x} & 5 & 10 & 4 & 5 \\
\hline
\end{tabular}
\caption{Number of iterations in Case 3}
\label{tab:pow_3}
\end{minipage}\hfill
\begin{minipage}[b]{0.45\textwidth}
\centering
\begin{tabular}{|c|c|c|c|c|c|}
\hline
$x^0$ & $\;\fsx\;$ & $\;\fsy\;$  & BS & BC & CD \\
\hline
$-1000$ & 3 & 20 & 3 & 3 & 2 \\
\hline
$-100$ & 3 & 18 & 3 & 3 & 2 \\
\hline
$-10$ & 3 & 7 & 3 & 3 & 2 \\
\hline
30 & 2 & 6 & 2 & 2 & 2 \\
\hline
3000 & 3 & 9 & 3 & 3 & 2 \\
\hline
$3\cdot 10^6$ & 3 & 15 & 3 & 3 & 3 \\
\hline
\end{tabular}
\caption{Number of iterations in Case 4}
\label{tab:pow_4}
\end{minipage}
\end{table}
\FloatBarrier

A few patterns stand out. Whenever Newton's methods on $\fsx$ and on $\fsy$ both converge, BS matches the better of the two, and BS still converges when Newton's method on $\fsx$ fails, as in Tables~\ref{tab:pow_2} and \ref{tab:pow_3}. In the latter situation, however, BS may need more iterations than Newton's method on $\fsy$, as in Case 2 and for $x^0=1000$ in Case 3. Neither BC nor CD ever requires more iterations than BS. Between these two methods, CD is the faster one in Case 1 (Table~\ref{tab:pow_1}), where it saves one or two iterations over BC at every starting point, and in Case 4 (Table~\ref{tab:pow_4}), where it is never slower and saves one iteration at four of the six starting points. The two methods are almost indistinguishable in Case 3 (Table~\ref{tab:pow_3}), where they differ by one iteration at two starting points, once in favor of each. BC is the better one in Case 2 (Table~\ref{tab:pow_2}), by one iteration at four of the six starting points. Minimizing the flatness number therefore pays off, but not uniformly, and the high-precision experiments below show that its effect is best seen in the asymptotic regime. Finally, Cases 1 and 2 (Tables~\ref{tab:pow_1} and \ref{tab:pow_2}) differ only by the value of $r$. In Case 2, the solution $x^*\approx 6.2\times 10^{-3}$ lies close to the point $x=0$ where $\fsx'$ blows up, and Newton's method on $\fsx$ fails from every starting point.

\paragraph*{High-precision computations.}
For a given set of parameters $(p, q, r, x^0)$, we plot $\log_{10}|\fst(\fs^k)|$ against the iteration number $k$ for each method, using the {\em mpmath} library for high-precision computations with the tolerance $ \varepsilon = 10^{-160}$.
\begin{figure}[htbp]
  \centering
  \begin{subfigure}[b]{0.49\textwidth}
    \centering
    \includegraphics[width=\textwidth]{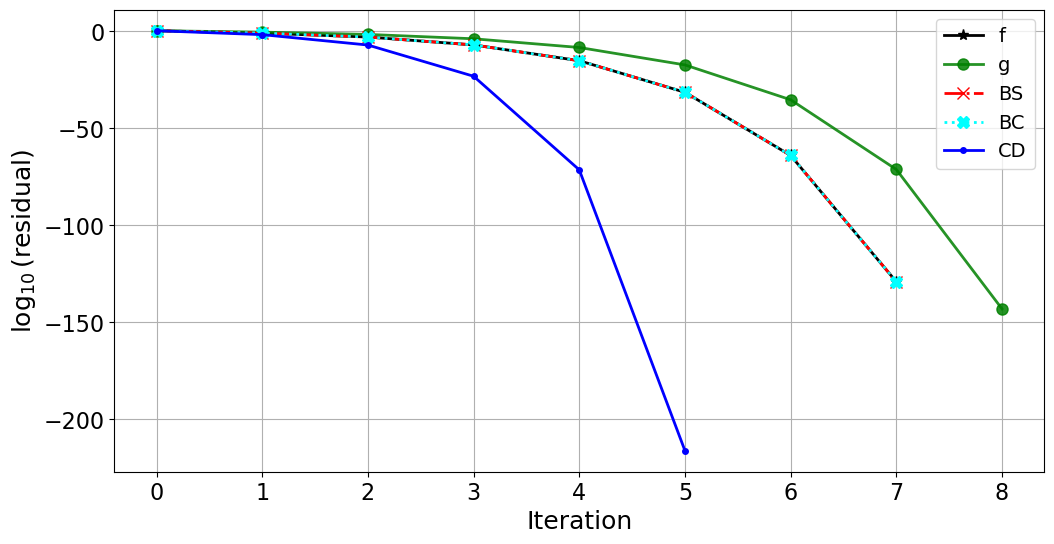}
    \caption{Case 1}
    \label{fig:spow.Case1}
  \end{subfigure}
  \hfill
  \begin{subfigure}[b]{0.49\textwidth}
    \centering
    \includegraphics[width=\textwidth]{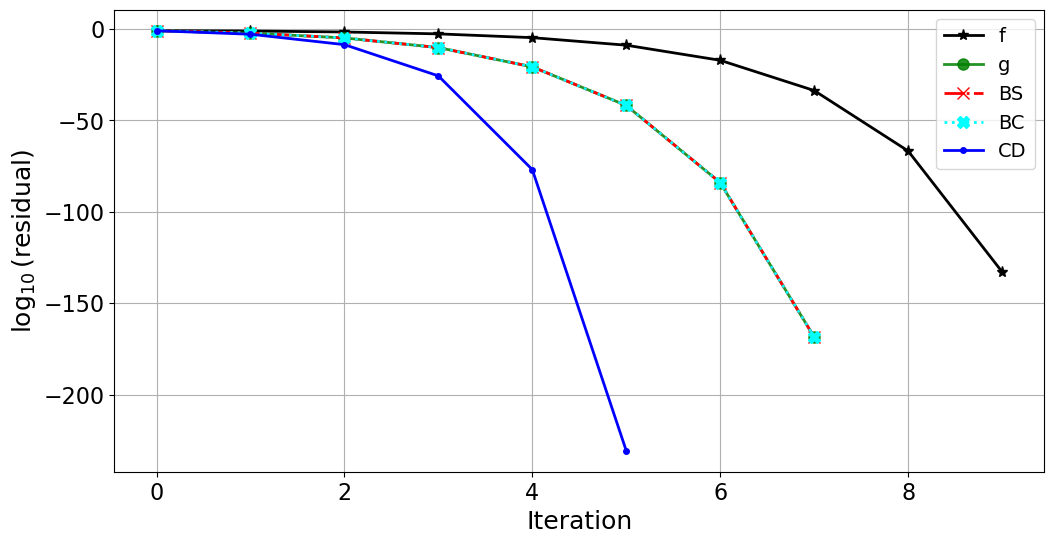}
    \caption{Case 2}
    \label{fig:spow.Case2}
  \end{subfigure}
  \vspace{0.5cm} 
  \begin{subfigure}[b]{0.49\textwidth}
    \centering
    \includegraphics[width=\textwidth]{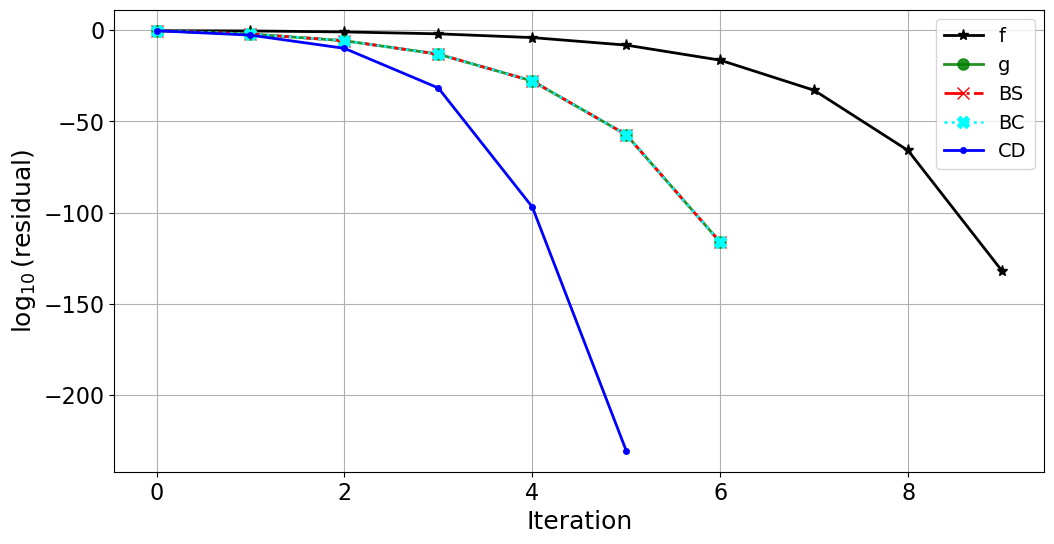}
    \caption{Case 3}
    \label{fig:spow.Case3}
  \end{subfigure}
  \hfill
  \begin{subfigure}[b]{0.49\textwidth}
    \centering
    \includegraphics[width=\textwidth]{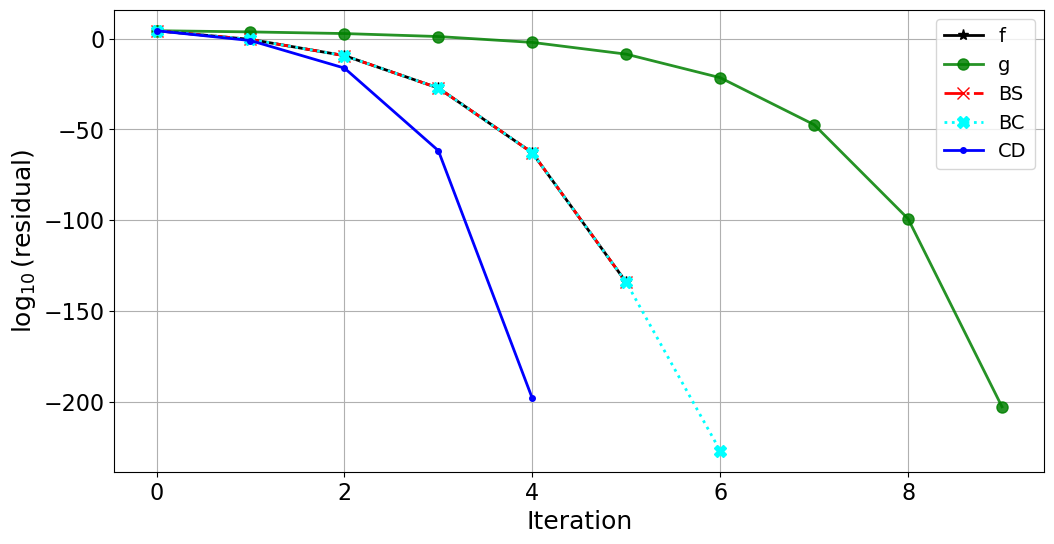}
    \caption{Case 4}
    \label{fig:spow.Case4}
  \end{subfigure}
  \caption{Evolution of residuals through iterations}
\end{figure}
\FloatBarrier
In the four cases, the curves of Newton's method on $\fsx$, BS and BC often overlap almost exactly, because these methods coincide in the corresponding regimes, as can be seen in Figures~\ref{fig:spow.Case1}--\ref{fig:spow.Case4}. The ordering of the curves reproduces the iteration counts of Tables~\ref{tab:pow_1}--\ref{tab:pow_4}. What the high-precision plots add is the behavior in the final iterations, where the CD curve is the steepest in all four cases. This asymptotic effect can be quantified as follows.

Since $\alpha^k$ cancels $\fst^{\prime\prime}_{\alpha^k}(\fs^k)$, a Taylor expansion of $h_{\alpha^k}$ around $\fs^k$ shows that the residual after one CD step behaves like $|\fst^{\prime\prime\prime}_{\alpha^k}(\fs^k)| |h_{\alpha^k}(\fs^k)|^3 / (6|\fst^{\prime}_{\alpha^k}(\fs^k)|^3)$. CD therefore converges {\em cubically}, whereas BC, being a Newton method with $\fst^{\prime\prime}\ne 0$ at the solution, converges quadratically. The residual sequences computed in Case 1 confirm this, with
\begin{align*}
    \log_{10}|\fst(\fs^{k+1})| - 3\log_{10}|\fst(\fs^k)| &\approx -2.107 \;\;\;\text{\rm for CD},\\
    \log_{10}|\fst(\fs^{k+1})| - 2\log_{10}|\fst(\fs^k)| &\approx -1.204 \;\;\;\text{\rm for BC},
\end{align*}
for sufficiently large $k$. These values agree with the theoretical asymptotic constants $\log_{10}(1/128)$ and $\log_{10}(1/16)$ obtained at the solution $x^*=1$ from the above expansion for CD and from $|\fsx^{\prime\prime}(x^*)| / (2\fsx^{\prime}(x^*)^2)$ for BC. Thus, for this scalar model, CD has a higher asymptotic order of convergence than BC. This does not always translate into a lower iteration count, since the cubic regime is reached only once the iterate is close enough to the solution. In Case 2, for instance, BC still finishes first at four of the six starting points.

\section{Preflattening a chemical system}\label{ch:chemistry}
We now turn to a vector-valued problem, the equilibrium of an ideal chemical mixture. The blow-up difficulty of the scalar case reappears here coordinate by coordinate, and we ask whether preflattening still pays off once several unknowns interact.

\subsection{Model}\label{sse:chem.model}
Nonlinear systems arise in numerous disciplines and exhibit a wide variety of difficulties\footnote{``Linear problems are all alike; every nonlinear problem is nonlinear in its own way,'' Mikhail Livschitz paraphrasing Leo Tolstoy's ``All happy families are alike; each unhappy family is unhappy in its own way.''}. This section is deliberately restricted to a specific class of problems, namely the equilibrium of ideal chemical mixtures. Such problems can be written in the abstract form
\begin{subequations}\label{eq:nnsystem}
\begin{alignat}{6}
    f_1:=& &\, a_{11}x_1 &\,+\,    &a_{12}x_2  &\,+\, &\dots &\,+\, &a_{1\nbsp}x_{\nbsp}  &\,-\, &b_1  &=0,\\
    f_2:=& &\, a_{21}x_1 &\,+\,    &a_{22}x_2  &\,+\, &\dots &\,+\, &a_{2\nbsp}x_{\nbsp}  &\,-\, &b_2  &=0,\\
    \vdots \; & & & \;\, \vdots    &  & \;\, \vdots & & \;\, \vdots & & \;\, \vdots & & \;\; \vdots \notag\\
    f_{\nbel}:=& &\, a_{\nbel 1}x_1 &\,+\,    &a_{\nbel 2}x_2  &\,+\, &\dots &\,+\, &a_{\nbel \nbsp}x_{\nbsp}  &\,-\, &b_{\nbel}  &=0,\\
    f_{\nbel+1}:=& &\, s_{11}\ln x_1    &\,+\,    &s_{12}\ln x_2     &\,+\, &\dots &\,+\, &s_{1\nbsp}\ln x_{\nbsp}    &\,-\, &t_1  &=0,\\
    f_{\nbel+2}:=& &\, s_{21}\ln x_1    &\,+\,    &s_{22}\ln x_2     &\,+\, &\dots &\,+\, &s_{2\nbsp}\ln x_{\nbsp}    &\,-\, &t_2  &=0,\\
    \vdots \; & & & \;\, \vdots   &  & \;\, \vdots & & \;\, \vdots & & \;\, \vdots & & \;\;\vdots \notag\\
    f_{\nbsp}:=& &\, s_{\nbre 1}\ln x_1    &\,+\,    &s_{\nbre 2}\ln x_2     &\,+\, &\dots &\,+\, &s_{\nbre\nbsp}\ln x_{\nbsp}    &\,-\, &t_{\nbre}  &=0,
\end{alignat}
\end{subequations}
in the unknowns $\boldsymbol{x} = (x_1,x_2,\dots, x_{\nbsp})\in (\mathbb{R}_+^*)^{\nbsp}$. For $i\in \{1,2,\dots,\nbsp\}$, $x_i$ represents the molar fraction of the $i$-th species, with the possible exception of the last component $x_\nbsp$. The first $\nbel$ equations are affine and express the material balances. The last $\nbsp - \nbel =: \nbre$ equations are nonlinear in the $x_i$'s but affine with respect to the $\ln x_i$'s, and they express the mass action law of each reaction that may occur in the system. The data $a_{ij}$ and $b_i$ are related to the so-called atomic matrix and to the feeds, while the data $s_{\ell j}$ and $t_\ell$ stem from the stoichiometric coefficients and from the reference chemical potentials. We refer the reader to \S\ref{sse:chemform} for a complete description of the variables, equations and data.

System \eqref{eq:nnsystem} can be given the condensed form
\begin{subequations}\label{eq:nnsystemcondensed}
    \begin{alignat}{3}
        A & \boldsymbol{x} & & - \, b & &= 0,\\
        S & \ln\boldsymbol{x} & & - \, t & & = 0,
    \end{alignat}
\end{subequations}
with $A \in \mathbb{R}^{\nbel\times\nbsp}$, $S\in\mathbb{R}^{\nbre\times\nbsp}$, $b \in \mathbb{R}^{\nbel}$ and $t \in \mathbb{R}^{\nbre}$, the logarithm acting componentwise. These data should not be confused with the generic matrix $A$ and vector $\rhsv$ of \S\ref{ch:notion}. Let $F = (f_1, f_2, \dots, f_{\nbsp})^{\!\top}$ denote the corresponding residual function and, for a vector $\boldsymbol{v}=(v_1,\dots,v_\nbsp)^{\!\top}$, let $\operatorname{Diag}(\boldsymbol{v})$ denote the diagonal matrix whose $i$-th diagonal entry is $v_i$. The Jacobian matrix of $F$ reads
\begin{equation}\label{eq:nablaF}
    \nabla F(\boldsymbol{x}) = \begin{bmatrix} A \\[2pt] S\operatorname{Diag}(\boldsymbol{x})^{-1} \end{bmatrix}.
\end{equation}
As $x_i \rightarrow 0^+$, the entries $s_{\ell i}/x_i$ of its $i$-th column blow up for every $\ell$ such that $s_{\ell i}\ne 0$, which may cause Newton's method to fail. Furthermore, the iterations break down as soon as one of the $x_i$'s becomes negative.

To prevent these two issues, the {\em log-trick} is commonly employed in the chemistry community. The idea is to solve, in the unknowns $\boldsymbol{y}=(y_1,y_2,\dots,y_{\nbsp})\in\mathbb{R}^{\nbsp}$,
\begin{equation}\label{eq:logtrick}
    G(\boldsymbol{y}) := \begin{bmatrix} A\exp\boldsymbol{y} - b \\[2pt] S\boldsymbol{y}-t \end{bmatrix} = 0,
\end{equation}
where $\exp$ acts componentwise. Writing $G=(g_1,g_2,\dots,g_{\nbsp})^{\!\top}$, the first $\nbel$ components $g_i$ are the material balances, now nonlinear, while the last $\nbre$ ones are the mass action laws, now affine. The log-trick therefore does not remove the nonlinearity but transfers it from one block of \eqref{eq:nnsystem} to the other. A similar difficulty then arises, since the Jacobian matrix
\begin{equation}\label{eq:nablaG}
    \nabla G(\boldsymbol{y}) = \begin{bmatrix} A\operatorname{Diag}(\exp\boldsymbol{y}) \\[2pt] S \end{bmatrix}
    = \begin{bmatrix} A\operatorname{Diag}(\boldsymbol{x}) \\[2pt] S \end{bmatrix}, \qquad \boldsymbol{x}=\exp\boldsymbol{y},
\end{equation}
has entries $a_{ij}x_j$ in its first $\nbel$ rows, and these entries blow up as $y_j\rightarrow +\infty$ whenever $a_{ij}\ne 0$.

If we are to guarantee robustness in all solution regimes, we must switch between $F$ and $G$ in a more refined way. To this end, following and extending \cite{JonvalBGCFT2025,Jonval2024}, we consider the intermediate system
\begin{equation}\label{eq:transitiontau}
    H(\boldsymbol{\fs}) := \begin{bmatrix} A\,\mathscr{X}(\boldsymbol{\fs}) - b \\[2pt]
    S\,\mathscr{Y}(\boldsymbol{\fs}) - t \end{bmatrix} = 0,
\end{equation}
in the unknowns $\boldsymbol{\fs} = (\fs_1, \fs_2, \dots,\fs_{\nbsp})$, where
$\mathscr{X}(\boldsymbol{\fs}) = (X_1(\fs_1), \dots, X_{\nbsp}(\fs_{\nbsp}))$ and
$\mathscr{Y}(\boldsymbol{\fs}) = (Y_1(\fs_1), \dots, Y_{\nbsp}(\fs_{\nbsp}))$. Writing $H=(h_1, h_2, \dots, h_{\nbsp})^{\!\top}$, the components read
\begin{equation}\label{eq:Hcomponents}
    h_i = \sum_{j=1}^{\nbsp} a_{ij} X_j(\fs_j) - b_i \;\; (1\le i\le \nbel),
    \qquad
    h_{\nbel+\ell} = \sum_{j=1}^{\nbsp} s_{\ell j} Y_j(\fs_j) - t_\ell \;\; (1\le \ell\le \nbre).
\end{equation}
The functions $\fs_i \mapsto (X_i(\fs_i),Y_i(\fs_i))$ are general parametrizations, which depend on the method under consideration. For now, we only require each $X_i$ to be increasing and twice differentiable, except possibly at isolated points, and
\begin{equation}\label{eq:logexp}
Y_i(\fs_i)= \ln X_i(\fs_i), \qquad \text{or equivalently} \qquad X_i(\fs_i) = \exp Y_i(\fs_i).
\end{equation}
Since
\begin{equation}\label{eq:HtauSyst}
    H(\boldsymbol{\fs}) = F(\mathscr{X}(\boldsymbol{\fs})) = G(\mathscr{Y}(\boldsymbol{\fs})),
\end{equation}
problem \eqref{eq:transitiontau} is a general right-preconditioned formulation of \eqref{eq:nnsystem} and \eqref{eq:logtrick}.

Furthermore, system \eqref{eq:transitiontau} is in the {\em separated variables form}, in the sense that $H$ is, up to a constant vector, a sum of functions of a single variable each,
\begin{equation}\label{eq:separevariform}
    H(\boldsymbol{\fs}) = \sum_{i=1}^{\nbsp} \mathscr{H}_i(\fs_i) - \begin{bmatrix} b\\ t \end{bmatrix}, \qquad
    \mathscr{H}_i(\fs_i) = \begin{bmatrix} X_i(\fs_i)\,A\bv_i \\[2pt] Y_i(\fs_i)\,S\bv_i \end{bmatrix},
\end{equation}
where $\bv_i$ is the $i$-th canonical basis vector of $\mathbb{R}^\nbsp$, so that $A\bv_i$ and $S\bv_i$ are the $i$-th columns of $A$ and $S$. This entails that all mixed second derivatives vanish,
\begin{equation}\label{eq:separated}
    \partial_{ij} H = 0 \qquad \text{for all } \; (i,j)\in\{1,2,\dots,\nbsp\}^2 \;\text{ with } i\ne  j,
\end{equation}
where $\partial_{ij}H$ denotes the vector with components $\partial_{ij} h_1,\dots,\partial_{ij} h_\nbsp$. A simpler formula can then be worked out for the flatness matrix $\mathcal{J}_H(\boldsymbol{\fs})$. From now on, we omit the argument $\boldsymbol{\fs}$ in various formulas whenever convenient.
\begin{lemma}\label{le:separatedform}
The flatness matrix of any twice differentiable function $H$ in the separated variables form is equal to
\begin{equation}\label{eq:separatedflatmat}
    \mathcal{J}_H = \nabla H^{-1} \mathfrak{D}^2 H \operatorname{Diag}(\nabla H^{-1} H)
\end{equation}
assuming $\nabla H^{-1}$ exists, where $(\mathfrak{D}^2 H)_{ij} = \partial_{jj} h_i$, and $\operatorname{Diag}(\boldsymbol{u})$
is the diagonal matrix whose $i$-th diagonal entry is $u_i$.
\end{lemma}
\begin{proof}
    See \S\ref{sse:SeparatedForm}.
\end{proof}

Applying Lemma~\ref{le:separatedform} to system \eqref{eq:transitiontau}, we obtain a closed-form expression for the flatness matrix $\mathcal{J}_H$. To state the result, we need some more notation. Assuming $\nabla G$ nonsingular and dropping its argument $\boldsymbol{y}$ for conciseness, let $r_{\ell i}$ denote the entry of $\nabla G^{-1}$ in row $i$ and column $\nbel + \ell$, that is,
\begin{equation}\label{eq:inversenablaG}
    \nabla G^{-1} \bv_{\nbel+\ell} = (r_{\ell 1}, r_{\ell 2}, \dots, r_{\ell\nbsp})^{\!\top},
    \qquad 1\le \ell\le\nbre.
\end{equation}
Note the reversed index order, the first index of $r_{\ell i}$ referring to the column and the second one to the row. Only these
$\nbre$ columns, associated with the reaction block of \eqref{eq:nablaG}, will appear below. Let us also consider the vector of {\em normalized residuals}
\begin{equation}\label{eq:trueres}
     \NLresnv = (\NLresns_1,\NLresns_2,\dots,\NLresns_{\nbsp})^{\!\top} := \nabla G(\boldsymbol{y})^{-1} G(\boldsymbol{y}) = \nabla G(\boldsymbol{y})^{-1} H(\boldsymbol{\fs}), \qquad \boldsymbol{y}=\mathscr{Y}(\boldsymbol{\fs}),
\end{equation}
which, up to sign, is the Newton increment of the log-trick formulation \eqref{eq:logtrick}.
Note that $\nabla G$, $\nabla G^{-1}$ and $\NLresnv$ take the same values at fixed $\boldsymbol{x}$ (or $\boldsymbol{y}$), whatever the parametrization. Finally, let
\begin{equation}\label{eq:PiProj}
    \Pi := \sum_{\ell=1}^{\nbre} \bv_{\nbel+\ell}\, \bv_{\nbel+\ell}^{\!\top}
\end{equation}
be the orthogonal projector onto the reaction block, that is, onto the span of the last
$\nbre$ canonical basis vectors of $\mathbb{R}^{\nbsp}$, and let $\mathscr{Y}' := (Y_1'(\fs_1),\dots,Y_\nbsp'(\fs_\nbsp))^{\!\top}$ and $\mathscr{Y}'' := (Y_1''(\fs_1),\dots,Y_\nbsp''(\fs_\nbsp))^{\!\top}$. We are now ready to write down the flatness matrix of system \eqref{eq:transitiontau}.

\begin{theorem}\label{th:flatmat00}
For all parametrizations $\fs_i \mapsto (X_i(\fs_i), Y_i(\fs_i))$ satisfying \eqref{eq:logexp}, the flatness matrix of system \eqref{eq:transitiontau} is given by
\begin{equation}\label{flat.mat00}
    \mathcal{J}_H = \Lambda^{-1}\big[\operatorname{Diag}(\boldsymbol{\aleph})
    - \nabla G^{-1}\Pi\,\nabla G\big]\,\Lambda\,\operatorname{Diag}(\NLresnv),
\end{equation}
where $\Lambda := \operatorname{Diag}(\mathscr{Y}')$, the projector $\Pi$ is defined in \eqref{eq:PiProj}, the vector $\NLresnv$ of normalized residuals in \eqref{eq:trueres}, and
\begin{equation}\label{eq:alephtaui}
    \aleph_i := 1+\dfrac{Y_i^{\prime\prime}}{(Y_i^{\prime})^2} = \frac{X''_i X_i}{(X'_i)^2}
    = \mathcal{J}_{X_i},
    \qquad \boldsymbol{\aleph}:=(\aleph_1,\dots,\aleph_{\nbsp})^{\!\top}.
\end{equation}
Moreover, $(\nabla G^{-1}\Pi\,\nabla G)_{ij} = \sum_{\ell=1}^{\nbre} r_{\ell i}s_{\ell j}$.
\end{theorem}
\begin{proof}
    See \S\ref{sse:flatchemgen}.
\end{proof}
\begin{remark}\label{rm:cancelcolumn}
Right-multiplication by $\operatorname{Diag}(\NLresnv)$ in \eqref{flat.mat00} implies that the $j$-th column of $\mathcal{J}_H$ contains the $j$-th normalized residual $\NLresns_j$ as a multiplicative factor. Therefore, if $\NLresns_j=0$, the whole $j$-th column of $\mathcal{J}_H$ vanishes. This will be useful in \S\ref{sse:chem.elim}--\S\ref{sse:chem.badchoice}.
\end{remark}
Formula \eqref{flat.mat00} also isolates the dependence on the parametrization, which enters only through $\boldsymbol{\aleph}$ and $\Lambda$, since at fixed $\boldsymbol{x}$ (or $\boldsymbol{y}$) the matrix $\nabla G^{-1}\Pi\nabla G$ and the vector $\NLresnv$ do not depend on it. Conjugating by $\Lambda$ removes the ratios $Y_j'/Y_i'$ and leaves
\begin{equation}\label{eq:Pm1JhP}
    \Lambda\,\mathcal{J}_H\,\Lambda^{-1}
    = \big[\operatorname{Diag}(\boldsymbol{\aleph}) - \nabla G^{-1}\Pi\,\nabla G\big]
      \operatorname{Diag}(\NLresnv),
\end{equation}
in which $\boldsymbol{\aleph}$ is the only quantity that depends on the parametrization. The Frobenius norm of this matrix supplies an upper bound of the flatness number $\rho_H$ that is easy to compute and easy to minimize. In the sequel, we refer to this upper bound as the {\em flatness estimate}.

\begin{prop}\label{pr:flatupperbound}
The flatness number of system \eqref{eq:transitiontau} is bounded from above by
\begin{equation}\label{eq:flatupperbound}
    \rho_H \le \| \Lambda\,\mathcal{J}_H\,\Lambda^{-1} \|_{\mathrm{Frob}}  =:  \widehat{\rho}_H.
\end{equation}
\end{prop}
\begin{proof}
Apply Proposition~\ref{pr:comprhonormJ} with the Frobenius norm and $P:=\Lambda^{-1}$, for which $P^{-1}\mathcal{J}_H P = \Lambda\,\mathcal{J}_H\,\Lambda^{-1}$.
\end{proof}

\subsection{Methods}\label{sse:chem.methods}
For each species, we may either keep $x_i$ or switch to $y_i=\ln x_i$, which gives $2^{\nbsp}$ possible Newton methods. Among these static choices, the fully log-tricked formulation \eqref{eq:logtrick} is the one most widely used by chemists, and exploring all $2^{\nbsp}$ combinations would be prohibitively expensive. We therefore build on \eqref{eq:logtrick} and consider the following methods.

BC applies the same scalar switch to every species, that is, $X_i=\textsf{X}$ and $Y_i=\textsf{Y}$ for all $i$ in \eqref{eq:transitiontau}, with
\begin{equation}\label{eq:bcsystem}
    (\textsf{X}(\fs), \textsf{Y}(\fs)) = \begin{cases}
        (\exp\fs, \fs) &\quad \text{if}\;\fs <0, \\
        (\fs+1,\ln(\fs+1)) &\quad \text{if}\; \fs\ge 0.
    \end{cases}
\end{equation}
As long as the iterates $\boldsymbol{x}^k=\exp\boldsymbol{y}^k$ of Newton's method on \eqref{eq:logtrick} remain in $(0,1)^\nbsp$, which is where molar fractions lie, BC produces exactly the same iterates, the affine branch of \eqref{eq:bcsystem} only acting as a safeguard against overshooting beyond $1$. This is why BC is regarded as the traditional method in \S\ref{sse:chem.results}.

CD assigns each species its own weight $\alpha_i \in [0,1]$ through
\begin{equation}\label{eq:CD33para}
        \textsf{X}_{\alpha}(\fs) = \left(1 + (1-\alpha) \fs\right)^{\frac{1}{1-\alpha}}, \qquad
        \textsf{Y}_{\alpha}(\fs) = \dfrac{1}{1-\alpha} \ln\left( 1 + (1-\alpha) \fs\right),
\end{equation}
defined for $\fs > -1/(1-\alpha)$ when $\alpha<1$, the values at $\alpha=1$ being set by continuity to $\textsf{X}_{1}(\fs) = \exp\fs$ and $\textsf{Y}_{1}(\fs)=\fs$, which are the limits obtained as $\alpha\to 1^{-}$. The weight $\alpha=0$ amounts to taking $x=1+\fs$ as the unknown, and $\alpha=1$ to taking $y=\ln x$. CD gathers these weights into the overall parametrization
\begin{subequations}\label{eq:overallparam}
\begin{align}
    \mathscr{X}_{\boldsymbol{\alpha}}(\fv) &= (\textsf{X}_{\alpha_1}(\tau_1), \textsf{X}_{\alpha_2}(\tau_2), \ldots, \textsf{X}_{\alpha_{\nbsp}}(\tau_{\nbsp})), \\
    \mathscr{Y}_{\boldsymbol{\alpha}}(\fv) &= (\textsf{Y}_{\alpha_1}(\tau_1), \textsf{Y}_{\alpha_2}(\tau_2),\dots, \textsf{Y}_{\alpha_{\nbsp}}(\tau_{\nbsp})),
\end{align}
\end{subequations}
indexed by $\boldsymbol{\alpha} = (\alpha_1,\alpha_2,\dots,\alpha_{\nbsp})\in \mathcal{A} = [0,1]^{\nbsp}$. The sans-serif letters $\textsf{X}$ and $\textsf{Y}$ denote single-species parametrizations shared by all species, namely the pair $\textsf{X},\textsf{Y}$ of BC and the pair $\textsf{X}_\alpha,\textsf{Y}_\alpha$ of CD, which carries the weight $\alpha$ as a subscript. The italic $X_i,Y_i$ of \eqref{eq:transitiontau} carry the species index instead, and for CD they are given by $X_i=\textsf{X}_{\alpha_i}$ and $Y_i=\textsf{Y}_{\alpha_i}$. Note also that the italic $X_\alpha,Y_\alpha$ of \eqref{CD.para} were indexed by a weight, since no species index was needed in \S\ref{ch:scalar}. Setting $H_{\boldsymbol{\alpha}}(\boldsymbol{\fs}) = F(\mathscr{X}_{\boldsymbol{\alpha}}(\boldsymbol{\fs})) = G(\mathscr{Y}_{\boldsymbol{\alpha}}(\boldsymbol{\fs}))$, we apply the right-preflattening strategy \eqref{eq:abstractalgo} to obtain the CD method. Given the current iterate $\boldsymbol{x}^k = (x_1^k,x_2^k,\dots, x_{\nbsp}^k)$, we compute
\begin{subequations}\label{CD.alg.system}
\begin{alignat}{2}
    \boldsymbol{\alpha}^k &= \text{arg} \textstyle\min_{\boldsymbol{\alpha}\in \mathcal{A}} \big\{ \rho_{H_{\boldsymbol{\alpha}}}(\mathscr{X}_{\boldsymbol{\alpha}}^{-1}(\boldsymbol{x}^k)) \big\}, & \qquad &\text{[tuning weights]}\label{eq:CDoptimstep}\\
    \fs^k_i &=
    \begin{cases}
        ((x_i^k)^{1-\alpha_i^k}-1)/(1-\alpha_i^k) &\;\; \text{\rm if }\, \alpha_i^k<1,\\
        \ln x_i^k &\;\; \text{\rm if }\, \alpha_i^k=1,
    \end{cases} &
    \qquad & \text{[current variable]}\label{eq:CDconvert}\\
    \fv^{k+1} & = \fv^k - \nabla H_{\boldsymbol{\alpha}^k}(\fv^k)^{-1}H_{\boldsymbol{\alpha}^k}(\fv^k), & \qquad &\text{[updated variable]}\label{eq:CDupdate}\\
    \boldsymbol{x}^{k+1} & = \mathscr{X}_{\boldsymbol{\alpha}^k}(\fv^{k+1}).& \qquad &\text{[primary unknown]}
\end{alignat}
\end{subequations}

The minimization step \eqref{eq:CDoptimstep} is the most delicate one. With the exact flatness number $\rho_{H_{\boldsymbol{\alpha}}}$, it can be carried out neither analytically nor numerically at a reasonable cost. With the upper bound \eqref{eq:flatupperbound} instead, it has the explicit solution given in Proposition~\ref{prop:CDoptimalpha} below, and this is the version used in practice. The second line of \eqref{eq:CDconvert} is the limit of the first one as $\alpha_i^k\to1^{-}$, and it inverts $\textsf{X}_{\alpha_i^k}(\fs)=\exp\fs$, the value taken at $\alpha_i^k=1$ by the convention adopted in \eqref{eq:CD33para}. In Proposition~\ref{prop:CDoptimalpha}, as in the proof of Proposition~\ref{pr:flatupperbound}, $P=\Lambda^{-1}=\operatorname{Diag}(\mathscr{Y}_{\boldsymbol{\alpha}}')^{-1}$.

\begin{prop}\label{prop:CDoptimalpha}
At a fixed $\boldsymbol{x}\in (\mathbb{R}_+^*)^\nbsp$, the upper bound $\widehat{\rho}_{H_{\boldsymbol{\alpha}}} = \|P^{-1} \mathcal{J}_{H_{\boldsymbol{\alpha}}}P \|_{\mathrm{Frob}}$ 
of the flatness number achieves its minimum over $(\boldsymbol{\alpha}, \boldsymbol{\fs} )\in [0,1]^\nbsp\times\mathbb{R}^\nbsp$ such that $\boldsymbol{x} = \mathscr{X}_{\boldsymbol{\alpha}}(\fv)$ at
\begin{equation}\label{form:alpha.optim}
    \alpha_i=\operatorname{Proj}_{[0,1]} \bigg( \sum_{\ell=1}^{\nbre}r_{\ell i} s_{\ell i} \bigg),
\end{equation}
where $\operatorname{Proj}_{[0,1]}$ stands for the projection operator on $[0,1]$.
\end{prop}
\begin{proof}
    See \S\ref{sse:minupperflatchem}.
\end{proof}

BC is robust, but its switching point is fixed and not tuned to minimize flatness. CD accelerates convergence near the solution, but it can diverge far from it. The {\em Hybrid} method combines the two. At each iterate, it computes $\boldsymbol{\alpha}^k$ from \eqref{form:alpha.optim}, then performs a BC step if $\widehat{\rho}_{H_{\boldsymbol{\alpha}^k}}>\mathfrak{a}$ and a CD step otherwise, the threshold being fixed to $\mathfrak{a}=0.5$ throughout. It also falls back to a BC step whenever the CD update leaves the domain of definition of \eqref{eq:CD33para}.

On the NaCl and Seawater test cases of \S\ref{sse:chem.results}, the gain brought by Hybrid over BC is modest, at most three iterations, and even zero at one starting point for NaCl and at two for Seawater. This is much less decisive than in the scalar case of \S\ref{ch:scalar}. Only on the smallest test case, H$_2$O, does the gain become substantial, Hybrid more than halving the BC count at five of the six starting points of Table~\ref{tab:res.H2O.outer} for the value $\nu=10^{-7}$ of the charge imbalance parameter \eqref{eq:nu.def}.

\subsection{Elimination as a preconditioner}\label{sse:chem.elim}
In many chemical systems, convergence is slowed down by {\em unbalanced nonlinearities}, that is, by a small subset of unknowns whose scale differs by orders of magnitude from that of the others\footnote{In water dissociation, $x_{\mathrm{H_2O}}$ is close to $1$ while $x_{\mathrm{H^+}}$ and $x_{\mathrm{OH^-}}$ can be smaller by six orders of magnitude or more.}. Lanzkron et al.~\cite{LanzkronRoseWilkes1996} proposed to eliminate these degrees of freedom directly. Let the index set $\{1,\ldots,\nbsp\}$ be partitioned into a ``good'' subset and a ``bad'' subset $\mathcal B$, whose selection is postponed to \S\ref{sse:chem.badchoice}. We split the unknowns accordingly into $\boldsymbol{\fs} = (\boldsymbol{\fs}_g, \boldsymbol{\fs}_b)$, with $\boldsymbol\fs_b=(\fs_j)_{j\in\mathcal B}$, and the equations into $H = (H_g, H_b)$. More generally, the subscript $b$ attached to a vector of $\mathbb{R}^{\nbsp}$, as in $\boldsymbol{v}_b$ or $\{\boldsymbol{v}\}_b$, designates its subvector indexed by $\mathcal{B}$. If $H_b(\boldsymbol{\fs}_g, \boldsymbol{\fs}_b) = 0$ has a unique solution $\boldsymbol{\fs}_b = \aef_b(\boldsymbol{\fs}_g)$ for every $\boldsymbol{\fs}_g$, substituting it into $H_g$ gives a reduced system in $\boldsymbol{\fs}_g$ alone. Newton's method on this reduced system involves a Schur complement Jacobian, which is costly to form. Hwang et al.~\cite{HwangLinCai2010,HwangSuCai2015} showed that it is equivalent to a cheaper and more convenient two-step process, which reads, at iterate $\boldsymbol{\fs}^k$,
\begin{enumerate}
    \item {\sc projection}. Solve $H_b(\boldsymbol{\fs}_g^k,\widetilde{\boldsymbol{\fs}}_b^k) = 0$ for $\widetilde{\boldsymbol{\fs}}_b^k$, and set $\widetilde{\boldsymbol{\fs}}^k = \aef(\boldsymbol{\fs}^k) := (\boldsymbol{\fs}_g^k,\widetilde{\boldsymbol{\fs}}_b^k)$.
    \item {\sc linearization}. Perform one global Newton step,
    \[
    \boldsymbol{\fs}^{k+1} = \widetilde{\boldsymbol{\fs}}^k - \nabla H(\widetilde{\boldsymbol{\fs}}^k)^{-1}H(\widetilde{\boldsymbol{\fs}}^k).
    \]
\end{enumerate}
This avoids the Schur complement, and it shows that nonlinear elimination is itself a right-preconditioner~\cite{brenner2025equivalence}, since the residual function $H$ is untouched and only the unknowns are substituted through $\aef$. In the setting of the monotone Newton theorem, the two-step process moreover retains the global monotone convergence of Newton's method, including for semismooth systems, provided that the (generalized) Jacobian matrix is an M-matrix~\cite{BrennerMarelliTran2026}.

Unlike the PDE-based systems considered in~\cite{brenner2025equivalence,BrennerMarelliTran2026}, where the mesh naturally pairs unknowns with equations, our system has no such structure. The normalized residuals \eqref{eq:trueres} restore this correspondence, because $\NLresns_j$ is, up to the factor $-1/Y_j'$, the Newton increment of the $j$-th unknown. Projecting onto them, that is, solving
\begin{equation}\label{eq:projectNLres}
    \NLresnv_b(\boldsymbol{\fs}_g^k, \widetilde{\boldsymbol{\fs}}_b^k) = 0,
\end{equation}
would therefore be natural, but $\NLresnv$ is far more nonlinear than $H$, which would make this projection expensive. We use instead a frozen version, the {\em linear normalized residuals} at a reference state $\overline{\boldsymbol{y}}$,
\begin{equation}\label{eq:fakeresidual}
    \resnv ( \overline{\boldsymbol{y}}; \boldsymbol{\fs}) := \nabla G( \overline{\boldsymbol{y}})^{-1} H(\boldsymbol{\fs}).
\end{equation}
By contrast with $\resnv$, we refer to $\NLresnv(\boldsymbol\fs) = \nabla G(\mathscr{Y}(\boldsymbol\fs))^{-1}H(\boldsymbol\fs)$, in which the normalizing Jacobian is re-evaluated at the current state instead of being frozen, as the {\em nonlinear normalized residuals}. Taking $\overline{\boldsymbol{y}} = \boldsymbol{y}^k := \mathscr{Y}(\boldsymbol{\fs}^k)$, we redefine the elimination map accordingly. From now on, $\aef_b(\boldsymbol{\fs}_g)$ denotes the solution of $\resnv_b(\boldsymbol{y}^k; \boldsymbol{\fs}_g, \cdot\,) = 0$ rather than that of $H_b(\boldsymbol{\fs}_g,\cdot\,)=0$, and the two steps become
\begin{subequations}
\begin{align}
    \resnv_b(\boldsymbol{y}^k; \boldsymbol{\fs}_g^k, \widetilde{\boldsymbol{\fs}}_b^k) & = 0, \label{eq:projection}\\
    \boldsymbol{\fs}^{k+1} & = \widetilde{\boldsymbol{\fs}}^k - \nabla \resnv(\boldsymbol{y}^k; \widetilde{\boldsymbol{\fs}}^k)^{-1} \resnv(\boldsymbol{y}^k; \widetilde{\boldsymbol{\fs}}^k),\label{eq:linearization}
\end{align}
\end{subequations}
where $\nabla\resnv$ denotes the Jacobian matrix with respect to $\boldsymbol{\fs}$. The following key fact makes this substitution legitimate.
\begin{theorem}\label{th:equivNLresiduals}
The solution of the projection step \eqref{eq:projection} also satisfies \eqref{eq:projectNLres}, and vice-versa.
\end{theorem}
\begin{proof}
    See \S\ref{proof:simultvanishRes}.
\end{proof}

Solving the less expensive frozen problem \eqref{eq:projection} thus cancels the nonlinear normalized residuals $\NLresnv_b$ as well. Moreover, since $\resnv(\boldsymbol{y}^k;\cdot\,)$ differs from $H$ by the constant invertible factor $\nabla G(\boldsymbol{y}^k)^{-1}$, the linearization step can be performed on $H$ itself, by the affine invariance of Newton's method recalled in \S\ref{sse:notion}. The algorithm actually implemented is therefore the following.
\begin{prop}\label{pr:elimproc}
    The two-step procedure \eqref{eq:projection}--\eqref{eq:linearization} is equivalent to
    \begin{enumerate}
        \item Solve for $\widetilde{\fv}^k_b$ the {\em inner} system 
    \begin{equation}\label{eq:innersys}
        \big\{ \nabla G(\boldsymbol{y}^k)^{-1} H(\fv_g^k, \widetilde{\fv}_b^k) \big\}_b =0 
    \end{equation}
    and set the projected state as $\widetilde{\boldsymbol{\tau}}^{k} = (\boldsymbol{\tau}^k_g, \widetilde{\fv}^k_b )$.
    \item Perform one global {\em outer} Newton iteration
    \begin{equation}\label{eq:linearEH}
    \boldsymbol{\tau}^{k+1} = \widetilde{\boldsymbol{\tau}}^{k}- \nabla H(\widetilde{\boldsymbol{\tau}}^{k})^{-1} H(\widetilde{\boldsymbol{\tau}}^{k}).
    \end{equation}
    \end{enumerate}
\end{prop}
\begin{proof}
Step 1 is \eqref{eq:projection} written out. Since $\nabla\resnv(\boldsymbol{y}^k;\boldsymbol{\fs}) = \nabla G(\boldsymbol{y}^k)^{-1}\nabla H(\boldsymbol{\fs})$, the factor $\nabla G(\boldsymbol{y}^k)^{-1}$ cancels out in \eqref{eq:linearization}, which is therefore identical to \eqref{eq:linearEH}.
\end{proof}
Note that $\nabla H = \nabla G(\mathscr{Y}(\boldsymbol{\fs}))\,\Lambda$ by \eqref{eq:logexp} and the chain rule, so that $\nabla H^{-1}H = \Lambda^{-1}\NLresnv$. Since $\NLresnv_b(\widetilde{\boldsymbol{\fs}}^k)=0$ by Theorem~\ref{th:equivNLresiduals}, the outer step \eqref{eq:linearEH} leaves the bad unknowns unchanged, that is, $\boldsymbol{\fs}^{k+1}_b=\widetilde{\boldsymbol{\fs}}^k_b$. The bad unknowns are thus updated by the projection step alone.

The inner system is again in the separated variables form, each of its equations being an affine combination of the $X_j(\fs_j)$ and $Y_j(\fs_j)$ with $j\in\mathcal B$, so that BC, CD and Hybrid carry over to it. The optimal CD weights for the inner problem take the same form as \eqref{form:alpha.optim}, with the entries of the inverse Jacobian matrix of the inner system with respect to the variables $y_j$, $j\in\mathcal B$, and the coefficients of its terms $Y_j(\fs_j)$ in place of $r_{\ell i}$ and $s_{\ell i}$. In our tests, BC always converged on the inner problem, although sometimes slowly, CD was faster but less robust, and Hybrid with threshold $0.5$ was consistently the best inner solver. We therefore use Hybrid for the inner problem throughout the rest of this section.

One subtlety remains for the outer step. Chaining \eqref{eq:innersys} and \eqref{eq:linearEH} defines the fixed-point map $\Phi_{H_{\boldsymbol{\alpha}}} \circ \aef$. Assuming that the projection step has a unique solution, Theorem~\ref{th:equivNLresiduals} shows that $\aef_b(\boldsymbol{\fs}_g)$ is characterized by $\NLresnv_b(\boldsymbol{\fs}_g,\aef_b(\boldsymbol{\fs}_g))=0$. Hence $\aef_b(\boldsymbol{\fs}_g)$ does not depend on $\boldsymbol{\fs}_b$, even though the frozen matrix $\nabla G(\boldsymbol{y}^k)$ does, and the columns of $\nabla\aef$ indexed by $\mathcal{B}$ vanish. By the chain rule, the Jacobian matrix of the composite map at $\boldsymbol{\fs}^k$ is $\mathcal{J}_{H_{\boldsymbol{\alpha}}}(\widetilde{\boldsymbol{\fs}}^k) \nabla\aef(\boldsymbol{\fs}^k)$. On the other hand, since $\NLresnv_b(\widetilde{\boldsymbol{\fs}}^k) =0$, Remark \ref{rm:cancelcolumn} implies that the columns of $\mathcal{J}_{H_{\boldsymbol{\alpha}}}(\widetilde{\boldsymbol{\fs}}^k)$ indexed by $\mathcal{B}$ vanish as well. Writing both matrices in block form according to the good and bad unknowns, and recalling that $\aef(\boldsymbol{\fs}) = (\boldsymbol{\fs}_g,\aef_b(\boldsymbol{\fs}_g))$, one readily checks that the product reduces to $\mathcal{J}_{H_{\boldsymbol{\alpha}}}(\widetilde{\boldsymbol{\fs}}^k)$. The Jacobian matrix of the composite map is thus the ordinary flatness matrix of $H_{\boldsymbol{\alpha}}$ evaluated at the projected state $\widetilde{\boldsymbol{\fs}}^k$, and not at the entrance state $\boldsymbol{\fs}^k$. Consequently, the CD weights for the outer step must be computed from \eqref{form:alpha.optim} at the projected state $\widetilde{\boldsymbol{x}}^k$ (equivalently $\widetilde{\boldsymbol{y}}^k$), and not at the entrance state $\boldsymbol{x}^k$. There is no circularity here, since the inner system \eqref{eq:innersys} can equally be written in terms of $\widetilde{\boldsymbol{x}}^k$ or $\widetilde{\boldsymbol{y}}^k$. The projected state is therefore defined independently of any parametrization, and the weights $\boldsymbol{\alpha}^k$ are only needed to convert it into $\widetilde{\boldsymbol{\fs}}^k$.

Elimination is thus a right-preconditioner in the sense of \S\ref{sse:def.preflat}. What \S\ref{sse:chem.badchoice} adds is a selection of the bad subset, based on the column norms of the matrix \eqref{eq:Pm1JhP}, that is designed to reduce the flatness estimate as well, a property confirmed by the experiments of \S\ref{sse:chem.results}. In our view, it is this second property, rather than the preconditioning itself, that makes the combination with CD effective.

\subsection{Choice of the bad subset}\label{sse:chem.badchoice}
For H$_2$O, the normalized residuals have two components that dominate the third one, so that $\mathcal{B}=\{2,3\}$ is a natural choice, which remains valid throughout the iterations. For NaCl and Seawater, the orders of magnitude of the normalized residuals change from one iteration to the next, so that no fixed choice is appropriate and the bad subset has to be selected automatically at each iteration. We now derive a systematic, threshold-free rule from the flatness estimate. Let $\widecheck\rho_H:=\widehat\rho_{H_{\boldsymbol\alpha}}$ denote the bound of Proposition~\ref{pr:flatupperbound} evaluated at the optimal weights $\boldsymbol\alpha$ of \eqref{form:alpha.optim}, which only depends on the current state. Substituting these weights into \eqref{eq:Pm1JhP}, the squared Frobenius norm splits column by column as
\begin{equation}\label{eq:sumofGamma}
    \widecheck{\rho}_H^2(\boldsymbol{\fs}) = \sum_{j=1}^{\nbsp}\Gamma_j, \qquad
    \Gamma_j = \bigg\{ \beta_j^2 + \sum_{1\le i\ne j\le \nbsp} \Big( \sum_{\ell=1}^{\nbre} r_{\ell i}s_{\ell j}\Big)^2 \bigg\}  |\NLresns_j|^2,
\end{equation}
with $\beta_j=\operatorname{Proj}_{[0,1]}(\sum_\ell r_{\ell j}s_{\ell j})-\sum_\ell r_{\ell j}s_{\ell j}$. In other words, $\Gamma_j$ is the squared Euclidean norm of the $j$-th column of the matrix \eqref{eq:Pm1JhP} at the optimal weights. Several standard tools borrowed from statistics and machine learning were tested on the family $(\Gamma_j)_{1\le j\le\nbsp}$ without success, such as $2$-means and other segmentation methods \cite{Hartigan1975}, robust scalings like the interquartile range (IQR) or the median absolute deviation (MAD) \cite{Tukey1977}, and spacing methods \cite{DavidNagaraja2003}. Most of them hinge on a user-defined threshold or tuning constant. As for $2$-means, although free of any threshold, it always returns two nonempty clusters and therefore declares a nonempty bad subset even when all the $\Gamma_j$'s are equal, which is precisely a situation in which no elimination should be performed.

Since the flatness that matters for the two-step procedure is the one at the projected state (see \S\ref{sse:chem.elim}), we wish $\widecheck{\rho}_H(\widetilde{\boldsymbol{\fs}})$ to be much smaller than $\widecheck{\rho}_H(\boldsymbol{\fs})$.

\paragraph*{Decrease the flatness estimate by cancelling the largest columns.}
The same decomposition holds at the projected state, $\widecheck{\rho}_H^2(\widetilde{\fv}) = \sum_{j=1}^{\nbsp}\widetilde{\Gamma}_j$, where $\widetilde{\Gamma}_j$ is obtained from $\Gamma_j$ by evaluating every state-dependent quantity at the projected state, that is, by replacing $\beta_j$, $r_{\ell i}$ and $\NLresns_j$ by $\widetilde{\beta}_j$, $\widetilde{r}_{\ell i}$ and $\widetilde{\NLresns}_j$. The stoichiometric coefficients $s_{\ell j}$ are data of the problem and therefore carry no tilde. In particular, $\widetilde{\Gamma}_j$ still contains $|\widetilde{\NLresns}_j|^2$ as a multiplicative factor. The projection step \eqref{eq:projection} cancels the component $\widetilde{\resns}_j$ of $\resnv(\boldsymbol{y}^k;\widetilde{\boldsymbol{\fs}}^k)$ for every $j\in\mathcal{B}$, so that $\widetilde{\NLresns}_j = 0$ for every $j\in\mathcal{B}$ by Theorem~\ref{th:equivNLresiduals}. For these indices, $\widetilde{\Gamma}_j = 0$, and the $j$-th column no longer contributes to the flatness estimate at the projected state. This observation suggests taking as elements of the bad subset $\mathcal{B}$ those indices $j$ whose $\Gamma_j$'s, taken together, contribute most to $\widecheck{\rho}_H^2 (\fv)$.

To put this idea into practice, we need to clarify the meaning of ``taken together'' and ``contribute most''. First, let us sort the column contributions in nonincreasing order,
\begin{equation}\label{normcol.sorted}
    \Gamma_{(1)} \ge \Gamma_{(2)} \ge \ldots \ge \Gamma_{(\nbsp)},
\end{equation}
where $j\mapsto (j)$ is a permutation of $\{1,\ldots,\nbsp\}$, the subscript $(j)$ designating the index of the $j$-th largest contribution. To detect the group of most salient columns, let us consider the cumulative sums
\begin{equation}\label{eq:cumulativeSum}
    R_j = \sum_{\ell=1}^{j} \, \Gamma_{(\ell)}, \qquad R_0:=0.
\end{equation}
We have $R_1 = \Gamma_{(1)}$ and $R_{\nbsp} = \widecheck{\rho}_H^2$. The sequence $(R_j)_{0\le j\le\nbsp}$ is nondecreasing and concave, in the sense that $R_{j+2} - 2R_{j+1} + R_j = \Gamma_{(j+2)}-\Gamma_{(j+1)}\le 0$, as illustrated in Figure~\ref{fig:normcol}, where the blue curve is a continuous interpolation of the discrete data $(j,R_j)$. The question is to find an index $j_*$ as small as possible such that the partial sum $R_{j_*}$ is close enough to the total $\widecheck{\rho}^2_H$.

\paragraph*{Algorithm with threshold: norms of columns (NC).}
For a given threshold $\varepsilon\in (0,1)$, let
\begin{subequations}
    \begin{align}
        j_*(\varepsilon) & = \min\{ j \, | \,R_j\ge (1-\varepsilon^2)\widecheck{\rho}^2_H\},\label{eq:goodenough}\\
        \mathcal{B}(\varepsilon) & =\{ (1), \, (2), \, \dots, \, (j_*(\varepsilon))\}.
    \end{align}
\end{subequations}
In other words, $j_*(\varepsilon)$ is the first index for which the partial sum $R_j$ reaches $\widecheck{\rho}^2_H$ within a relative error of $\varepsilon^2$. The ``bad'' unknowns are those whose rank in the ordering \eqref{normcol.sorted} does not exceed $j_*(\varepsilon)$. The {\em raison d'\^etre} of \eqref{eq:goodenough} is the following: the cancellation of the $j_*(\varepsilon)$ strongest columns has erased a relative portion of at least $1-\varepsilon^2$ from the initial squared Frobenius norm $\widecheck{\rho}^2_H(\fv)$. Thus, if the values of $\widetilde{\Gamma}_j$ from the  $\nbsp - j_*(\varepsilon)$ remaining columns have not changed much from those of $\Gamma_j$, the new squared Frobenius norm is expected to satisfy
\begin{equation}
   \widecheck{\rho}^2_H(\widetilde{\fv}) \le \varepsilon^2 \widecheck{\rho}^2_H(\fv),\qquad
   \text{or equivalently, }\qquad \widecheck{\rho}_H(\widetilde{\fv}) \le \varepsilon \widecheck{\rho}_H(\fv).
\end{equation}

This naive algorithm suffers from a few drawbacks. First, it is always delicate to determine an optimal value for $\varepsilon$. If $\varepsilon$ is too big, some components that should be classified as ``bad'' may be overlooked; on the other hand, if $\varepsilon$ is excessively small, the ``bad'' subset is over-chosen. In addition, $\varepsilon$ is an easily mutable quantity, which means that it could change across various test cases and at different starting points. In a nutshell, there is no universal good choice for $\varepsilon$.
\begin{figure}[h]
  \centering
    \includegraphics[width=0.8\textwidth]{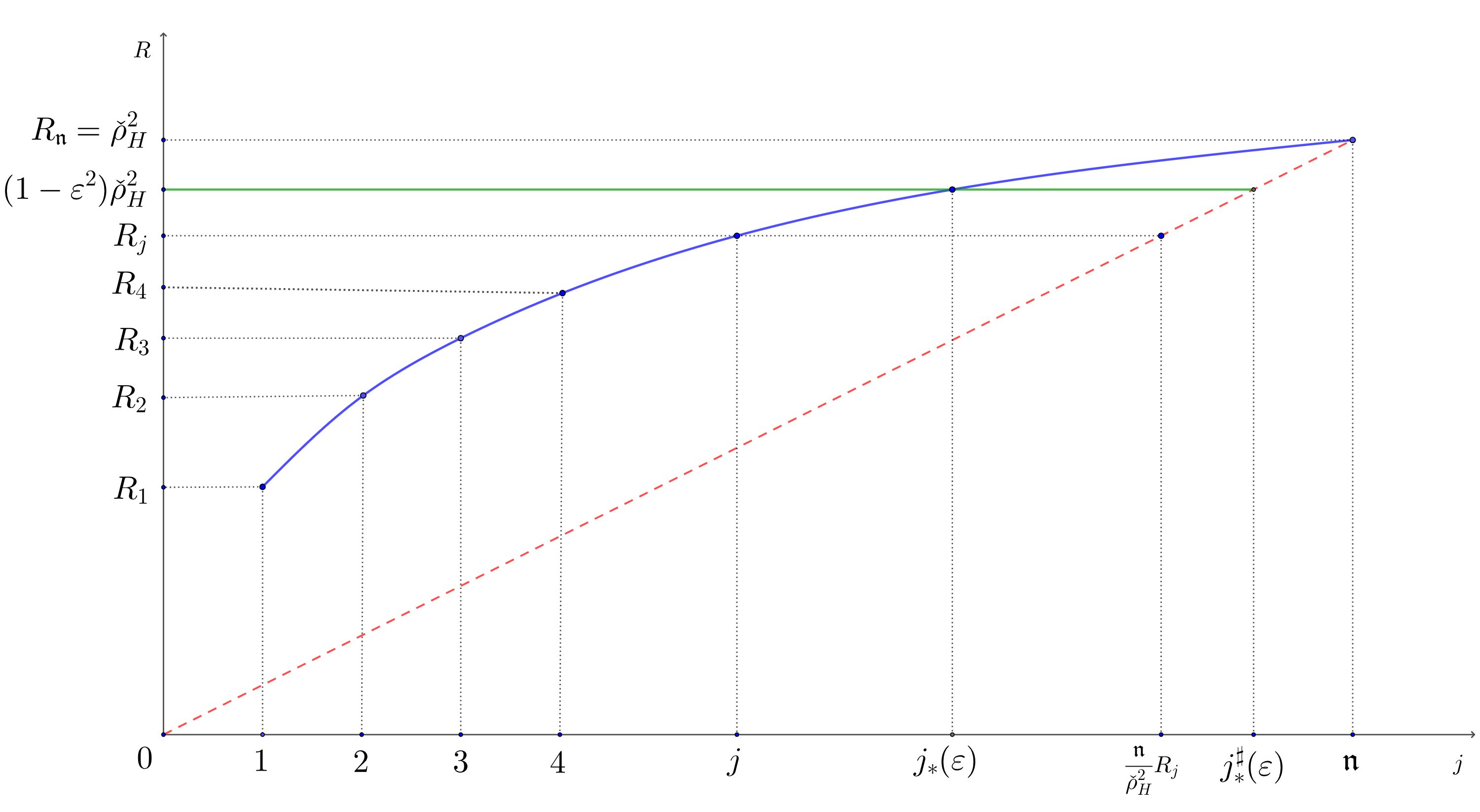}
  \caption{Trajectory of $R_j$}
  \label{fig:normcol}
\end{figure} 
\FloatBarrier

\paragraph*{Algorithm without threshold: above-average norms of columns (AANC).}
This situation prompts us to work out a threshold-free rule. Geometrically, imposing a fixed threshold $\varepsilon$ through \eqref{eq:goodenough} amounts to drawing the horizontal line $R = (1-\varepsilon^2)\widecheck{\rho}^2_H$, represented by the green line in Figure~\ref{fig:normcol}, and to looking for its intersection with the blue curve joining the points $(j, R_j)$. For convenience, assume that this blue curve is the graph of a continuous, increasing and concave function $\mathscr{R}$ on $[0,\nbsp]$ such that $\mathscr{R}(j) = R_j$. The abscissa of the intersection is then
\begin{equation}\label{eq:jstar}
    j_*(\varepsilon) = \mathscr{R}^{-1} \big( (1 -
                        \varepsilon^2)\widecheck{\rho}^2_H \big),
\end{equation}
where $j_*(\varepsilon)$ is temporarily allowed to take non-integer values.

The worst situation occurs when the sequence $(R_j)_{0\le j\le \nbsp}$ is arithmetic, or equivalently when the $\Gamma_j$'s are all equal. In this case, $\mathscr{R}(j) = (j/\nbsp) \widecheck{\rho}^2_H$, and the desired level is only reached at
\begin{equation}\label{eq:jstarsharp}
    j_*^\sharp(\varepsilon) = (1- \varepsilon^2) \nbsp,
\end{equation}
as depicted by the red line in Figure~\ref{fig:normcol}. In general, $\mathscr{R}$ is concave with $\mathscr{R}(0)=0$ and $\mathscr{R}(\nbsp)=\widecheck{\rho}^2_H$, so that $\mathscr{R}(j)\ge (j/\nbsp)\widecheck{\rho}^2_H$ on $[0,\nbsp]$. Therefore, $j_*^\sharp(\varepsilon) \ge j_*(\varepsilon)$ for every such function $\mathscr{R}$.

In the absence of a universally optimal fixed value of $\varepsilon$, we choose the most economical one, namely the one that achieves the largest gap between the two abscissas \eqref{eq:jstar} and \eqref{eq:jstarsharp}. In formal language, we want to select
\begin{equation}\label{eq:maxecartj}
  \varepsilon_* = \operatorname{arg} \max_{\varepsilon\in [0,1]} \{ j_*^\sharp(\varepsilon) - j_*(\varepsilon) \}
\end{equation}
as the solution of a continuous optimization problem. To return to the discrete level, we restrict the admissible values of $\varepsilon$ to those for which $(1-\varepsilon^2) \widecheck{\rho}^2_H$ coincides with one of the $R_j$'s, that is, $\varepsilon \in \{ \varepsilon_0, \, \varepsilon_1, \,\ldots, \, \varepsilon_\nbsp \}$ with
\begin{equation}
    \varepsilon_j = \bigg( 1 - \frac{R_j}{\widecheck{\rho}^2_H}\bigg)^{\!1/2}.
\end{equation}
Assuming the $\Gamma_j$'s to be positive, we have $j_*^\sharp(\varepsilon_j) = \nbsp R_j/\widecheck{\rho}^2_H$ and $j_*(\varepsilon_j) = j$, so that the discrete version of \eqref{eq:maxecartj} reads
\begin{equation}\label{eq:maxecartjDISC}
    j_*= \operatorname{arg} \max_{j\in\{0,1,\ldots,\nbsp\}} \bigg\{\dfrac{\nbsp}{\widecheck{\rho}^2_H} R_j - j \bigg\}.
\end{equation}
When this problem has several maximizers, we retain the smallest one, in keeping with our wish to take $j_*$ as small as possible.
\begin{prop}\label{pr:solminjnc}
    The solution of the discrete maximization problem \eqref{eq:maxecartjDISC} is
    \begin{equation}\label{eq:solmaxecart}
    j_*= \max \bigg\{ j \mid \Gamma_{(j)} > \frac{\widecheck{\rho}^2_H}{\nbsp} \bigg\}
    = \max \bigg\{ j \mid \Gamma_{(j)} > \frac{1}{\nbsp} \sum_{\ell=1}^{\nbsp} \, \Gamma_{\ell} \, \bigg\},
\end{equation}
and the corresponding ``bad'' subset is $\mathcal{B} = \{(1),\, (2), \, \ldots,\, (j_*)\}$.
\end{prop}
\begin{proof}
    See \S\ref{sse:badAANC}.
\end{proof}

Here, we use the convention $\max\emptyset := 0$. The set in \eqref{eq:solmaxecart} is empty only when all the $\Gamma_j$'s are equal, in which case $j_*=0$, $\mathcal B=\emptyset$ and no elimination is performed. Since $\widecheck{\rho}^2_H/\nbsp$ is the average of the $\Gamma_\ell$'s by \eqref{eq:sumofGamma}, rule \eqref{eq:solmaxecart} simply places the indices with above-average contributions in the bad subset, that is, $\mathcal{B}=\{ j \mid \Gamma_j > \widecheck{\rho}^2_H/\nbsp\}$, whence its name {\em ``above-average norms of columns''} (AANC). In particular, no sorting is needed in practice.

\paragraph*{Double AANC.}
Preliminary numerical tests show that, in most cases, AANC outperforms the rules mentioned above in terms of the complexity measure \eqref{eq:complexity} introduced in \S\ref{sse:chem.results}. There are, however, cases in which it incurs a higher cost. Inspecting the iterations reveals that, with the single-pass rule \eqref{eq:solmaxecart}, the flatness estimate after the projection step is occasionally larger than before it. This defeats the purpose of our construction and suggests that the selected bad subset misses some inherently bad components. As a remedy, we perform an additional round of selection by AANC on the remaining components. The resulting rule, called {\em double AANC}, reads
\begin{subequations}\label{eq:double.AANC}
    \begin{align}
        \mathcal{B}_1 & = \Big\{ 1\le j \le \nbsp \mid \Gamma_j >  \frac{1}{\nbsp} \sum_{\ell=1}^{\nbsp} \, \Gamma_{\ell} \Big\},\\
        \mathcal{B}_2 & = \Big\{ j\not\in \mathcal{B}_1 \mid \Gamma_j >  \frac{1}{\nbsp - |\mathcal{B}_1|} \sum_{\ell\not\in\mathcal{B}_1} \, \Gamma_{\ell} \Big\},\\
        \mathcal{B} & = \mathcal{B}_1\cup\mathcal{B}_2.
    \end{align}
\end{subequations}

\subsection{Numerical results}\label{sse:chem.results}
We conduct numerical experiments on three test cases, with increasing numbers of species, elements and reactions,
\begin{itemize}
    \item[$\rhd$] \textbf{H\(_2\)O}, with 3 species, 2 elements and 1 reaction;
    \item[$\rhd$] \textbf{NaCl}, with 8 species, 4 elements and 4 reactions;
    \item[$\rhd$] \textbf{Seawater}, with 37 species, 10 elements and 27 reactions.
\end{itemize}
Full details of these test cases are given in \S\ref{sse:testcases}. The stopping criteria are $\|\resnv_b\|_2\le 10^{-10}$ for the inner solve and $\|H\|_2\le 10^{-10}$ for the outer step, both being capped at $100$ iterations. For BC and Hybrid, which involve no inner solve, the outer iterations are simply their Newton iterations. BCe and CDe denote the BC and CD formulations augmented with elimination, the inner problem being solved by Hybrid in both cases.

Starting points are built from initial mole numbers $\textsf{n}_i^0$ through
$\textsf{x}_i^0 = \textsf{n}_i^0/(\textsf{n}_1^0+\ldots+\textsf{n}_{|\mathcal{C}|}^0)$ for $1\le i\le|\mathcal{C}|$, together with $\omega^0 = 1/(\textsf{n}_1^0+\ldots+\textsf{n}_{|\mathcal{C}|}^0)$, where $|\mathcal{C}|$ is the number of species, $\textsf{x}_i$ the molar fraction of the $i$-th species and $\omega$ the reciprocal of the total number of moles (see \S\ref{sse:chemform}). We take $\textsf{n}_1^0 = 55$ for water, the remaining $\textsf{n}_j^0$ being reported in the tables. To compare the methods fairly, we track the overall complexity
\begin{equation}\label{eq:complexity}
\begin{split}
    \operatorname{complexity} = {} & \sum_{k}|\mathcal{B}^k|^2\times(\text{\rm inner iterations at outer step } k)\\
    & + \nbsp^2\times(\text{\rm outer iterations}),
\end{split}
\end{equation}
where $\mathcal{B}^k$ is the bad subset selected at outer step $k$. In this measure, one Newton iteration on a system of size $s$ is counted as $s^2$, and the first sum is absent for BC and Hybrid.

\paragraph*{H$_2$O.} Here, $\nbsp=3$ and the parameter $\nu$ defined in \eqref{eq:nu.def} measures the charge imbalance between $\mathrm{H^+}$ and $\mathrm{OH^-}$ (see \S\ref{sse:descript.water}). The bad subset is always $\mathcal{B}=\{2,3\}$. Two regimes were tested, $\nu=10^{-1}$ and $\nu=10^{-7}$, and both are reported in Tables~\ref{tab:res.H2O.outer} and \ref{tab:res.H2O.complexity}. Each of these two regimes brings out a different ingredient of the method. The regime $\nu=10^{-7}$ brings out the contribution of elimination itself, since BCe and CDe need only $1$ or $2$ outer iterations there, against $9$ to $25$ for the traditional method BC, which is an improvement of roughly one order of magnitude in the number of iterations. The regime $\nu=10^{-1}$, where the outer problem retains a curvature that elimination alone does not remove, brings out instead the specific contribution of preflattening, since BCe and CDe then differ markedly although they rely on the same elimination procedure.

\begin{table}[h]
\centering
\setlength{\tabcolsep}{3pt}
\begin{minipage}[b]{0.48\textwidth}
\centering
\begin{tabular}{|c|c|c|c|c|c|}
\hline
\multicolumn{2}{|c|}{Initial guess} & & & & \\ \cline{1-2}
    $\textsf{n}_2^0$ & $\textsf{n}_3^0$ & BC & Hybrid & BCe & CDe \\ \hline
    $10^{-2}$ & $2\cdot 10^{-2}$ & 7 & 4 & 4 & 2 \\ \hline
    $10^{-2}$ & $5\cdot 10^{-2}$ & 7 & 4 & 4 & 2 \\ \hline
    $10^{-4}$ & $2\cdot 10^{-4}$ & 7 & 4 & 4 & 2 \\ \hline
    $10^{-4}$ & $5\cdot 10^{-4}$ & 7 & 4 & 4 & 2 \\ \hline
    $\varepsilon_{32}$ & $2\varepsilon_{32}$ & 7 & 4 & 4 & 2 \\ \hline
    $\varepsilon_{32}$ & $5\varepsilon_{32}$ & 7 & 4 & 4 & 2 \\ \hline
\end{tabular}
\caption*{$\nu=10^{-1}$}
\end{minipage}
\hspace{0.015\textwidth}
\begin{minipage}[b]{0.48\textwidth}
\centering
\begin{tabular}{|c|c|c|c|c|c|}
\hline
\multicolumn{2}{|c|}{Initial guess} & & & & \\ \cline{1-2}
    $\textsf{n}_2^0$ & $\textsf{n}_3^0$ & BC & Hybrid & BCe & CDe \\ \hline
    $10^{-2}$ & $2\cdot 10^{-2}$ & 21 & 8  & 2 & 2 \\ \hline
    $10^{-2}$ & $5\cdot 10^{-2}$ & 25 & 10 & 2 & 2 \\ \hline
    $10^{-4}$ & $2\cdot 10^{-4}$ & 9  & 6  & 1 & 1 \\ \hline
    $10^{-4}$ & $5\cdot 10^{-4}$ & 22 & 9  & 1 & 1 \\ \hline
    $\varepsilon_{32}$ & $2\varepsilon_{32}$ & 19 & 6 & 1 & 1 \\ \hline
    $\varepsilon_{32}$ & $5\varepsilon_{32}$ & 14 & 5 & 1 & 1 \\ \hline
\end{tabular}
\caption*{$\nu=10^{-7}$}
\end{minipage}
\caption{Number of outer iterations for H$_2$O 
($\varepsilon_{32}=1.1920929\times 10^{-7}$ is the single-precision machine epsilon)}
\label{tab:res.H2O.outer}
\end{table}
\FloatBarrier

\begin{table}[h]
\centering
\setlength{\tabcolsep}{3pt}
\begin{minipage}[b]{0.48\textwidth}
\centering
\begin{tabular}{|c|c|c|c|c|c|}
\hline
\multicolumn{2}{|c|}{Initial guess} & & & & \\ \cline{1-2}
    $\textsf{n}_2^0$ & $\textsf{n}_3^0$ & BC & Hybrid & BCe & CDe \\ \hline
    $10^{-2}$ & $2\cdot 10^{-2}$ & 63 & 36 & 48 & 30 \\ \hline
    $10^{-2}$ & $5\cdot 10^{-2}$ & 63 & 36 & 48 & 30 \\ \hline
    $10^{-4}$ & $2\cdot 10^{-4}$ & 63 & 36 & 48 & 30 \\ \hline
    $10^{-4}$ & $5\cdot 10^{-4}$ & 63 & 36 & 48 & 30 \\ \hline
    $\varepsilon_{32}$ & $2\varepsilon_{32}$ & 63 & 36 & 48 & 30 \\ \hline
    $\varepsilon_{32}$ & $5\varepsilon_{32}$ & 63 & 36 & 48 & 30 \\ \hline
\end{tabular}
\caption*{$\nu=10^{-1}$}
\end{minipage}
\hspace{0.015\textwidth}
\begin{minipage}[b]{0.48\textwidth}
\centering
\begin{tabular}{|c|c|c|c|c|c|}
\hline
\multicolumn{2}{|c|}{Initial guess} & & & & \\ \cline{1-2}
$\textsf{n}_2^0$ & $\textsf{n}_3^0$ & BC & Hybrid & BCe & CDe \\ \hline
$10^{-2}$ & $2\cdot 10^{-2}$ & 189 & 72 & 50 & 50 \\ \hline
$10^{-2}$ & $5\cdot 10^{-2}$ & 225 & 90 & 70 & 70 \\ \hline
$10^{-4}$ & $2\cdot 10^{-4}$ & 81  & 54 & 45 & 45 \\ \hline
$10^{-4}$ & $5\cdot 10^{-4}$ & 198 & 81 & 41 & 41 \\ \hline
$\varepsilon_{32}$ & $2\varepsilon_{32}$ & 171 & 54 & 25 & 25 \\ \hline
$\varepsilon_{32}$ & $5\varepsilon_{32}$ & 126 & 45 & 25 & 25 \\ \hline
\end{tabular}
\caption*{$\nu=10^{-7}$}
\end{minipage}
\caption{Complexity of the methods for H$_2$O}
\label{tab:res.H2O.complexity}
\end{table}
\FloatBarrier

For $\nu=10^{-1}$, each algorithm needs the same number of iterations at all six starting points, namely $7$ for BC, $4$ for Hybrid, $4$ for BCe and $2$ for CDe. BCe and CDe use the same elimination procedure, the same bad subset and the same inner solver, and they differ only in the parametrization of the outer step. The twofold gap between them can therefore be ascribed to preflattening, and the mechanism can be made explicit. Since $\NLresnv_b$ vanishes at the projected state and the first column of $\nabla G$ is $(0,x_1,1)^{\!\top}$, the residual at the projected state reads $H(\widetilde{\boldsymbol{\fs}}^k)=\widetilde{\NLresns}_1\,(0,\widetilde{x}_1^k,1)^{\!\top}$. The charge balance $f_1=0$ thus holds exactly at the projected state, and it pins $x_2$ near $2\nu/(1+\nu)\approx 0.18$, because $x_3$ is negligible (about $1.5\times 10^{-17}$ at equilibrium).  The outer step then only updates the water fraction $x_1$, which has to decrease from $x_1^0\approx 1$ to $x_1^*\approx 0.82$. At the projected state, formula \eqref{form:alpha.optim} returns $\alpha_1=0$ for water, so that CDe performs this step in the variable $x_1$ itself, with respect to which the material balance $f_2$ is affine. This balance is then satisfied after a single outer step, and the next projection removes the remaining residual. BCe performs the same step in the variable $\ln x_1$, with respect to which $f_2$ is curved, and it only converges quadratically from an initial gap of about $0.18$, hence its four outer iterations. The complexity figures confirm this reading, with $63$ for BC, $36$ for Hybrid, $48$ for BCe and $30$ for CDe. With $\nbsp^2=9$, the outer steps account for $36$ units of the complexity of BCe and $18$ of that of CDe, so that both methods spend the same $12$ units in the inner solves and the whole gap between them comes from the outer step. Finally, Hybrid, which performs no elimination, needs as many iterations as BCe at a lower cost ($36$ against $48$). Elimination alone therefore does not guarantee a gain once the cost of the inner solves is accounted for, and the best result is obtained by combining it with preflattening.

The picture is markedly different for $\nu=10^{-7}$. There, BCe and CDe coincide at every starting point and need only $1$ or $2$ outer iterations, while BC needs between $9$ and $25$ and Hybrid between $5$ and $10$. The outer step again only updates the water fraction, but it now starts very close to the solution. At equilibrium, $x_2\approx 2.0\times 10^{-7}$ and $x_3 \approx 1.7\times 10^{-11}$, so that $x_1^*$ differs from $1$ by about $2\times 10^{-7}$, whereas $x_1^0 = 55 / (55 + \textsf{n}_2^0 + \textsf{n}_3^0)$ lies within about $1.1\times 10^{-3}$ of $1$. 

One outer step reduces the residual to about half the square of the gap $x_1^0 - x_1^*$, which is below the tolerance $10^{-10}$ as soon as this gap is smaller than about $1.4\times 10^{-5}$, that is, for the four starting points with $\textsf{n}_2^0 \le 10^{-4}$. For the other two starting points, a second outer step is needed. This holds for both parametrizations of the outer step, which explains why BCe and CDe coincide. High-precision residual tracking confirms that BCe and CDe converge quadratically near the solution.

Taken together, the two regimes complement each other. At $\nu=10^{-7}$, elimination resolves most of the difficulty, reducing the number of outer iterations of BC by a factor between $9$ and $22$ and the complexity by a factor between $1.8$ and $6.8$. At $\nu=10^{-1}$, elimination alone (BCe) and preflattening alone (Hybrid) bring the same reduction in iterations, from $7$ to $4$, and only their combination, CDe, halves this count again. Preflattening is thus a useful complement to elimination, precisely in those regimes where elimination alone does not flatten the outer problem completely.

\paragraph*{NaCl.} Here, $\nbsp=9$, the bad subset changes from one iteration to the next, and CDe is run with double AANC. Table~\ref{tab:res.NaCl} shows the resulting gains. Compared with BC, CDe reduces the number of outer iterations by a factor of up to $10$ (from $30$ to $3$ for $(\textsf{n}_1^0, \textsf{n}_{j\ne 1}^0) = (55, 10^{-4})$). It reduces the complexity by a factor of more than $3$ in this hardest case (from $2430$ to $729$), and by a factor between $1.2$ and $1.9$ at the other starting points. Compared with Hybrid, the complexity gain is marginal at the two starting points with $\textsf{n}_{j\ne 1}^0\in \{55,1\}$ ($801$ and $809$ against $810$), and lies between $1.2$ and $3$ at the other ones.

\begin{table}[h]
\begin{minipage}[b]{0.48\textwidth}
\centering
\begin{tabular}{|c|c|c|c|c|}
\hline
\multicolumn{2}{|c|}{Initial guess} &  &  &  \\ \cline{1-2}
    $\textsf{n}_1^0$ & $\textsf{n}_{j\ne 1}^0$ & BC & Hybrid & CDe \\ \hline
    $55$ & $55$        & 12 & 10 & 5 \\ \hline
    $55$ & $1$         & 12 & 10 & 3 \\ \hline
    $55$ & $0.1$       & 11 & 10 & 2 \\ \hline
    $55$ & $10^{-2}$   & 12 & 11 & 2 \\ \hline
    $55$ & $10^{-4}$   & 30 & 27 & 3 \\ \hline
    $55$ & $\varepsilon_{32}$ & 16 & 16 & 3 \\ \hline
\end{tabular}
\caption*{Outer iterations}
\end{minipage}
\hfill
\begin{minipage}[b]{0.48\textwidth}
\centering
\begin{tabular}{|c|c|c|c|c|}
\hline
\multicolumn{2}{|c|}{Initial guess} &  &  &  \\ \cline{1-2}
$\textsf{n}_1^0$ & $\textsf{n}_{j\ne 1}^0$ & BC & Hybrid & CDe \\ \hline
$55$ & $55$      & 972  & 810  & 801 \\ \hline
$55$ & $1$       & 972  & 810  & 809 \\ \hline
$55$ & $0.1$     & 891  & 810  & 670 \\ \hline
$55$ & $10^{-2}$ & 972  & 891  & 661 \\ \hline
$55$ & $10^{-4}$ & 2430 & 2187 & 729 \\ \hline
$55$ & $\varepsilon_{32}$ & 1296 & 1296 & 679 \\ \hline
\end{tabular}
\caption*{Complexity}
\end{minipage}
\caption{NaCl, CDe run with double AANC}
\label{tab:res.NaCl}
\end{table}
\FloatBarrier

\paragraph*{Seawater.} Here, CDe is again run with the double AANC rule \eqref{eq:double.AANC}. Table~\ref{tab:res.Seawater} confirms the same trend at a larger scale. CDe converges in $5$ to $9$ outer iterations at all starting points, against $16$ to $29$ for BC. Its complexity is about $2.8$ times lower than those of BC and Hybrid for $\textsf{n}_{j\ne 1}^0=10^{-4}$ ($13743$ against $38988$). At the other starting points, it is $1.4$ to $1.7$ times lower than that of BC and $1.2$ to $1.6$ times lower than that of Hybrid. Tracking the flatness estimate across iterations, we find that it is lower right after the projection step than right before it, at every iteration of every run reported here. This confirms that elimination actually flattens the system rather than merely reducing its size.

\begin{table}[h]
\begin{minipage}[b]{0.48\textwidth}
\centering
\begin{tabular}{|c|c|c|c|c|}
\hline
\multicolumn{2}{|c|}{Initial guess} &  &  &  \\ \cline{1-2}
    $\textsf{n}_1^0$ & $\textsf{n}_{j\ne 1}^0$ & BC & Hybrid & CDe \\ \hline
    $55$ & $55$      & 22 & 22 & 7 \\ \hline
    $55$ & $1$       & 19 & 16 & 6 \\ \hline
    $55$ & $0.1$     & 17 & 14 & 5 \\ \hline
    $55$ & $10^{-2}$ & 16 & 14 & 5 \\ \hline
    $55$ & $10^{-4}$ & 27 & 27 & 5 \\ \hline
    $55$ & $\varepsilon_{32}$ & 29 & 27 & 9 \\ \hline
\end{tabular}
\caption*{Outer iterations}
\end{minipage}
\hfill
\begin{minipage}[b]{0.48\textwidth}
\centering
\begin{tabular}{|c|c|c|c|c|}
\hline
\multicolumn{2}{|c|}{Initial guess} &  &  &  \\ \cline{1-2}
    $\textsf{n}_1^0$ & $\textsf{n}_{j\ne 1}^0$ & BC & Hybrid & CDe \\ \hline
    $55$ & $55$      & 31768 & 31768 & 21921 \\ \hline
    $55$ & $1$       & 27436 & 23104 & 19322 \\ \hline
    $55$ & $0.1$     & 24548 & 20216 & 16206 \\ \hline
    $55$ & $10^{-2}$ & 23104 & 20216 & 13787 \\ \hline
    $55$ & $10^{-4}$ & 38988 & 38988 & 13743 \\ \hline
    $55$ & $\varepsilon_{32}$ & 41876 & 38988 & 24693 \\ \hline
\end{tabular}
\caption*{Complexity}
\end{minipage}
\caption{Seawater, CDe run with double AANC}
\label{tab:res.Seawater}
\end{table}
\FloatBarrier

Together, these results qualify the conclusion of \S\ref{sse:chem.methods}. Preflattening alone is not enough for chemical systems, but combined with elimination and with a threshold-free selection of the bad subset, it becomes an effective accelerator. At every starting point of every test case, CDe needs fewer outer iterations than BC and Hybrid and has a lower complexity, although its complexity gain over Hybrid is marginal at two starting points for NaCl. On H$_2$O, the only test case for which BCe is reported, CDe outperforms BCe for $\nu=10^{-1}$ and matches it for $\nu=10^{-7}$.

\section{Conclusion}
We have introduced the ``flatness'' number, defined as the spectral radius of the Jacobian matrix of Newton's iteration function, to serve as a measure of the local nonlinearity of an algebraic system. This quantity is affine-invariant and vanishes for affine systems, which makes it a natural nonlinear counterpart of the condition number.  We have called ``preflattening'' any right change of variables, chosen anew at each iteration, that reduces this number or an upper bound of it.

On a scalar model problem, the resulting Continuous Dynamic (CD) method cancels the flatness number at each iterate and converges cubically. For chemical equilibrium systems, preflattening alone brings only modest gains. Combined with nonlinear elimination, and with a bad subset selected by the threshold-free double AANC rule derived from the flatness estimate, it reduces the number of outer iterations of the traditional BC method by a factor of up to $10$ on NaCl and up to about $5$ on Seawater, and the overall complexity by a factor of up to about $3$ on both test cases. Extending this approach to non-ideal mixtures and to nonlinear systems arising from the discretization of PDEs is left for future work.

\bibliographystyle{siamplain}
\bibliography{Preflatten}

\appendix
\section{Proofs of Section~\ref{ch:notion}}
\subsection{Column-wise expression of the flatness matrix}\label{sse:flatMATexp}
Differentiating $\Phi_F(\ov) = \ov - \nabla F(\ov)^{-1} F(\ov)$ with respect to the $j$-th component of $\ov$, we get
\begin{equation}
   j\mbox{-th column of } \mathcal{J}_F(\ov) = {\bv}_j - \partial_j [\nabla F(\ov)^{-1}] F(\ov) - \nabla F(\ov)^{-1} \partial_j F(\ov),
\end{equation}
where ${\bv}_j$ stands for the $j$-th vector of the canonical basis. Since $\nabla F(\ov)^{-1} \nabla F(\ov) = I$ implies $\nabla F(\ov)^{-1} \partial_j F(\ov) = {\bv}_j$, cancellation with the first summand gives
\begin{equation}\label{eq:jthcolumn2}
   j\mbox{-th column of } \mathcal{J}_F(\ov) = - \partial_j [\nabla F(\ov)^{-1}] F(\ov).
\end{equation}
Differentiating $\nabla F(\ov)^{-1} \nabla F(\ov) = I$ with respect to $\ovc_j$, we obtain
\begin{equation}
    \partial_j [\nabla F(\ov)^{-1}] \nabla F(\ov) + \nabla F(\ov)^{-1} \partial_j [\nabla F(\ov)] = 0, 
\end{equation}
from which we infer that
\begin{equation}\label{eq:djinvert}
    \partial_j [\nabla F(\ov)^{-1}] = - \nabla F(\ov)^{-1} \partial_j [\nabla F(\ov)] \nabla F(\ov)^{-1}.
\end{equation}
Plugging \eqref{eq:djinvert} into \eqref{eq:jthcolumn2}, we end up with the claim \eqref{eq:flatMATjth}.

\subsection{Affine invariance of the flatness number}\label{sse:invarianceJrho}
Differentiating ${}_A F(\ov) = A F(\ov)$, we get $\nabla {}_A F(\ov) = A \nabla F(\ov)$. Hence,
\begin{equation}
    \nabla {}_A F(\ov)^{-1} = \nabla F(\ov)^{-1} A^{-1}, \qquad \nabla {}_A F(\ov)^{-1} {}_A F(\ov) = \nabla F(\ov)^{-1} F(\ov).
\end{equation}
By Proposition~\ref{pr:flatMATexp}, each column of $\mathcal{J}_{{}_A F}(\ov)$ is found to be equal to the corresponding column of $\mathcal{J}_{F}(\ov)$. Thus, \eqref{eq:invarianceLIN} holds true.

Differentiating $F_{A,\rhsv} (\ovw) = F(A\ovw+\rhsv)$ with respect to $\ovw$, we get $\nabla F_{A,\rhsv} (\ovw) = \nabla F(A\ovw+\rhsv) A$, where the $\nabla$ in the right-hand side is carried out with respect to $\ov$. Hence, with $\ov=A\ovw+\rhsv$,
\begin{equation}
    \nabla F_{A,\rhsv} (\ovw)^{-1} = A^{-1}\nabla F(\ov)^{-1}, \qquad
    \nabla F_{A,\rhsv} (\ovw)^{-1} F_{A,\rhsv} (\ovw) = A^{-1}\nabla F(\ov)^{-1} F(\ov).
\end{equation}
The iteration function for $F_{A,\rhsv}$ then reads
\begin{align*}
    \Phi_{F_{A,\rhsv}}(\ovw)  & = \ovw -   \nabla F_{A,\rhsv} (\ovw)^{-1} F_{A,\rhsv} (\ovw)\\
    &= A^{-1} (\ov - \rhsv) - A^{-1}\nabla F(\ov)^{-1} F(\ov)\\
    &= A^{-1} \Phi_F(\ov) - A^{-1}\rhsv.
\end{align*}
Differentiating both sides with respect to $\ovw$ yields $\mathcal{J}_{F_{A,\rhsv}}(\ovw) = A^{-1} \mathcal{J}_{F}(\ov) A$, from which \eqref{eq:invarianceAFF} follows.

\subsection{Upper bound for the flatness number}\label{sse:flatnumupper}
Since the spectrum is invariant by similarity transformation, we have
\begin{equation}
    \rho_F(\ov) = \rho(\mathcal{J}_F(\ov)) = \rho(P^{-1} \mathcal{J}_F(\ov) P).
\end{equation}
On the other hand, it is a well-known fact \cite[Theorem 5.6.9]{HornJohnson2013} that $\rho(M) \le \|M\|$ for every square matrix $M$ and every matrix norm. Hence, we have \eqref{eq:comprhonormJ} with $M=P^{-1} \mathcal{J}_F(\ov) P$.

\section{Proofs of Section~\ref{ch:scalar}}
\subsection{Smaller flatness number}\label{proof:rho.smaller}
The two functions considered are
\begin{equation}
    \fsx(x)=px+q\{x\}^{1/m}-r, \qquad
    \fsy(y)=p\{y\}^m+qy-r,
\end{equation}
which implies that
\begin{equation}
    \fsx^{\prime}(x)=p+q\textstyle\frac{1}{m}|x|^{1/m-1}, \qquad
    \fsy^{\prime}(y)=pm|y|^{m-1}+q
\end{equation}
and
\begin{equation}
    \fsx^{\prime\prime}(x)=q\textstyle\frac{1}{m}\left(\textstyle\frac{1}{m}-1\right)\{x\}^{1/m-2}, \qquad
    \fsy^{\prime\prime}(y)=pm(m-1)\{y\}^{m-2}.
\end{equation}
By \eqref{eq:flat.scalar}, the corresponding flatness numbers are
\begin{subequations}
\begin{align}
    \rho_{\fsx}(x) & =\left|\dfrac{(px+q\{x\}^{1/m}-r)q\textstyle\frac{1}{m}\left(\textstyle\frac{1}{m}-1\right)\{x\}^{1/m-2}}{\left(p+q\textstyle\frac{1}{m}|x|^{1/m-1}\right)^2}\right|\\
    \rho_{\fsy}(y) & =\left|\dfrac{(p\{y\}^m+qy-r)pm(m-1)\{y\}^{m-2}}{\left(pm|y|^{m-1}+q\right)^2}\right|.
\end{align}
\end{subequations}
Let $x\ne 0$ be such that $\fsx(x)\ne 0$, and let $y=\{x\}^{1/m}$. Since $|x|^{1/m-1}=|y|^{1-m}$ and $|x|^{1/m-2}=|y|^{1-2m}$, we have
\begin{equation}
    \fsx'(x)=\frac{\fsy'(y)}{m|y|^{m-1}}, \qquad |\fsx''(x)| = \frac{q(m-1)}{m^2}\,|y|^{1-2m}.
\end{equation}
Using $\fsx(x)=\fsy(y)$, we can rewrite the two flatness numbers as
\begin{equation}
    \rho_{\fsx}(x) = \frac{|\fsx(x)|\,(m-1)}{|y|\,\fsy'(y)^2}\, q, \qquad
    \rho_{\fsy}(y) = \frac{|\fsx(x)|\,(m-1)}{|y|\,\fsy'(y)^2}\, pm|y|^{m-1}.
\end{equation}
The common factor being positive, $\rho_{\fsy}(y)<\rho_{\fsx}(x)$ if and only if $pm|y|^{m-1}<q$, that is,
\begin{equation}
    |y|<\left(\dfrac{q}{pm}\right)^{1/(m-1)}.
\end{equation}
The proof is complete.

\subsection{Optimal weight for the scalar equation}\label{proof:opt.para}
Recall the CD parametrization
\begin{equation}
    X_\alpha(\fs)=\{\fs\}^{\frac{m}{(1-\alpha) m+\alpha}},\qquad Y_\alpha(\fs)=\{\fs\}^{\frac{1}{(1-\alpha) m+\alpha}}.
\end{equation}
Their first derivatives are
\begin{subequations}
\begin{alignat}{2}
    X^{\prime}_\alpha(\tau) & =\dfrac{m}{(1-\alpha) m+\alpha}|\tau|^{\frac{\alpha(m-1)}{(1-\alpha) m+\alpha}}
    & & =\dfrac{m}{(1-\alpha) m+\alpha}\left[\dfrac{X_{\alpha}(\tau)}{Y_{\alpha}(\tau)}\right]^{\alpha},\\
    Y^{\prime}_\alpha(\tau) & =\dfrac{1}{(1-\alpha) m+\alpha}|\tau|^{\frac{(\alpha-1)(m-1)}{(1-\alpha) m+\alpha}}
    & & =\dfrac{1}{(1-\alpha) m+\alpha}\left[\dfrac{Y_{\alpha}(\tau)}{X_{\alpha}(\tau)}\right]^{1-\alpha},
\end{alignat}
\end{subequations}
and the second derivatives are
\begin{subequations}
\begin{alignat}{2}
    X^{\prime \prime}_\alpha(\tau) & =\dfrac{\alpha m(m-1)}{[(1-\alpha) m+\alpha]^{2}}\{\tau\}^{\frac{(2 \alpha-1) m-2 \alpha}{(1-\alpha) m+\alpha}}
    & & =\dfrac{\alpha m(m-1)}{[(1-\alpha) m+\alpha]^{2}} \dfrac{X_{\alpha}(\tau)}{\tau^{2}},\\
    Y^{\prime \prime}_\alpha(\tau) & =\dfrac{(\alpha-1)(m-1)}{[(1-\alpha) m+\alpha]^{2}}\{\tau\} ^{\frac{(2 \alpha-2) m-(2 \alpha-1)}{(1-\alpha) m+\alpha}}
    & & =\dfrac{(\alpha-1)(m-1)}{[(1-\alpha) m+\alpha]^{2}} \dfrac{Y_{\alpha}(\tau)}{\tau^{2}}.
\end{alignat}
\end{subequations}
By \eqref{eq:halphatau}, we obtain
\begin{equation}
    \fst^{\prime \prime}_{\alpha}(\tau) =pX^{\prime \prime}_{\alpha}(\tau)+qY^{\prime \prime}_{\alpha}(\tau) =\dfrac{m-1}{[(1-\alpha) m+\alpha]^{2}}\left\{\alpha pm \dfrac{X_{\alpha}(\tau)}{\tau^{2}}+(\alpha-1) q\dfrac{Y_{\alpha}(\tau)}{\tau^{2}}\right\},
\end{equation}
which leads, for $\tau\neq0$, to
\begin{align*}
\fst^{\prime \prime}_{\alpha}(\tau)=0 & \;\;\Longleftrightarrow\;\; \alpha pmX_{\alpha}(\tau)-(1-\alpha) qY_{\alpha}(\tau)=0 \\
& \;\;\Longleftrightarrow\;\; \alpha\left[pmX_{\alpha}(\tau)+qY_{\alpha}(\tau)\right]=qY_{\alpha}(\tau)\\
& \;\;\Longleftrightarrow\;\; \alpha=\dfrac{q y}{pm x+qy}.
\end{align*}
The last equivalence holds because $x$ and $y$ are nonzero and have the same sign, so that $pmx+qy\neq0$. It is then straightforward to check that $\alpha\in [0,1]$, which completes the proof.

\section{Proofs of Section~\ref{ch:chemistry}}
\subsection{Physical origin}\label{sse:chemform}
We now provide a brief overview of the physical origin of the abstract model \eqref{eq:nnsystem}. Following classical textbooks such as \cite{Bethke2007,ErdiToth1989,SmithMissen1982}, we define a single-phase chemical system at fixed temperature and pressure as the collection $\{\mathcal{C}, \mathcal{E}, \mathcal{R}\}$ of three sets, namely
\begin{alignat*}{2}
    \mathcal{C} & = (\mathtt{C}_1, \mathtt{C}_2, \ldots, \mathtt{C}_{|\mathcal{C}|})  & & \qquad \text{a set of } |\mathcal{C}| \text{ chemical species};\\
    \mathcal{E} & = (\mathtt{E}_1, \mathtt{E}_2, \ldots, \mathtt{E}_{|\mathcal{E}|}) & & \qquad \text{a set of } |\mathcal{E}| < |\mathcal{C}| \text{ chemical elements};\\
    \mathcal{R} & = (\mathtt{R}_1, \mathtt{R}_2, \ldots, \mathtt{R}_{|\mathcal{R}|}) & & \qquad \text{a set of } |\mathcal{R}| = |\mathcal{C}|  - |\mathcal{E}| \text{ chemical reactions}.
\end{alignat*}
The set $\mathcal{E}$ contains all the elements that enter into the composition of the species of the set $\mathcal{C}$. The interaction between $\mathcal{E}$ and $\mathcal{C}$ is depicted by a $|\mathcal{E}|\times |\mathcal{C}|$-matrix $\textsf{A}$, called {\em atomic matrix}, whose entry $\textsf{A}_{ij}$ is the number of atoms of $\mathtt{E}_i$ involved in molecule $\mathtt{C}_j$\footnote{Our description here is deliberately simplified for the sake of clarity. In reality, the definition of a species is broader and may include atoms, molecules, ions, radicals, or complexes. In a multiphase setting, the same chemical entity in different phases (gas, liquid, solid) gives rise to different chemical species for thermodynamic description. For example, $\HHO (L),\, \HHO (G),\, \HHO (S)$ are considered distinct chemical species.}. The interaction between $\mathcal{C}$ and $\mathcal{R}$ is reflected by a $|\mathcal{R}|\times |\mathcal{C}|$-matrix $\textsf{S}^{\!\top}$, the transpose $\textsf{S}$ of which is called {\em stoichiometric matrix}. The entry $s_{\ell j}:=\textsf{S}^{\!\top}_{\ell j}$ is the stoichiometric coefficient of molecule $\mathtt{C}_j$ in reaction $\mathtt{R}_\ell$, which reads
\[
\sum_{\substack{1\le j\le |\mathcal{C}|\\ s_{\ell j} < 0}} (- s_{\ell j}) \mathtt{C}_j \rightleftharpoons
\sum_{\substack{1\le j\le |\mathcal{C}|\\ s_{\ell j} > 0}} s_{\ell j} \mathtt{C}_j.
\]
Note that stoichiometry implies
\begin{equation}
    \textsf{A} \textsf{S} = 0.
\end{equation}
As an example, the water dissociation system
\[
\mathcal{C} = ( \HHO, \, \Hp, \, \OHm ), \qquad \mathcal{E} = ( \mathrm{H}, \,\mathrm{O} ),
\qquad \mathcal{R} = ( \Hp + \OHm \rightleftharpoons \HHO )
\]
consists of 3 species, 2 elements and 1 reaction. The atomic and stoichiometric matrices are given by
\begin{align*}
  \textsf{A} & = \left[\begin{array}{ccc}
                2 & \;\;1 & \;\;\;1 \\
                1 & \;\;0 & \;\;\;1 \\
            \end{array}\right] \begin{array}{c}
                \mathrm{H}\\
                \mathrm{O}
            \end{array} \qquad\qquad  \textsf{S}^{\!\top} = [1, \, - 1, \, -1].
            \\*
            & \begin{array}{cccc}
                &  \; \HHO & \Hp & \!\!\OHm
            \end{array}
\end{align*}

The chemical equilibrium problem, also called {\em speciation problem}, consists in finding, given the number of moles $\textsf{f}_i$ of each element $\mathtt{E}_i$, the number of moles $\textsf{n}_j$ of each species $\mathtt{C}_j$ such that
\begin{subequations}\label{eq:speciation}
\begin{align}
            \textsf{A} \nmolevec &= \textsf{f},\label{eq:material}\\
            \textsf{S}^{\!\top} \boldsymbol{\mu}(\nmolevec) &= 0,\label{eq:thermodynamical}
\end{align}
\end{subequations}
where $\textsf{f} = (\textsf{f}_1, \textsf{f}_2,\ldots,\textsf{f}_{|\mathcal{E}|})$ is the known feed vector, $\nmolevec = (\nmole_1,\nmole_2, \ldots, \nmole_{|\mathcal{C}|})$ is the unknown composition vector and $\boldsymbol{\mu} (\nmolevec) = (\mu_1(\nmolevec), \mu_2(\nmolevec), \ldots, \mu_{|\mathcal{C}|}(\nmolevec))$ is the vector of chemical potentials. For an ideal mixture, the potentials take the form
\begin{equation}\label{eq:potential}
    \mu_j(\nmolevec) = \mu_j^{\circ} + \textsf{R}\textsf{T} \ln \frac{\nmole_j}{\nmole_1 + \nmole_2 + \ldots + \nmole_{|\mathcal{C}|}},
\end{equation}
where $\textsf{R}$ is the universal gas constant, $\textsf{T}$ the temperature and $\mu_j^\circ$ the reference potential of species $\mathtt{C}_j$ at that temperature. For non-ideal mixtures, formula \eqref{eq:potential} has to be supplemented by corrective terms. Anyhow, all admissible potentials $\mu_j$ must be $0$-homogeneous with respect to $\nmolevec$ and satisfy the Gibbs--Duhem conditions
\begin{equation}
    \sum_{j = 1}^{|\mathcal{C}|} \textsf{n}_j \frac{\partial \mu_j}{\partial \textsf{n}_i} (\nmolevec) = 0, \qquad \forall i\in\{1,\ldots,|\mathcal{C}|\}.
\end{equation}
In this work, we shall restrict ourselves to ideal mixtures.

In system \eqref{eq:speciation}, the first subset of equations \eqref{eq:material} expresses the conservation of elements, whereas the second subset \eqref{eq:thermodynamical} expresses the vanishing of the chemical affinity of each reaction at equilibrium. Note that system \eqref{eq:speciation}--\eqref{eq:potential} is only one possible approach to solving the speciation problem. Among many other approaches \cite{LealKSS2017}, the most prominent one seeks to directly minimize the Gibbs energy
\begin{equation}\label{eq:Gibbsmin}
    \min_{\textsf{A} \nmolevec =  \textsf{f}} \; \sum_{j=1}^{|\mathcal{C}|} {\nmole}_j \mu_j(\nmolevec)
\end{equation}
subject to the conservation constraint \eqref{eq:material}. It can be shown that the optimality conditions of the constrained minimization problem \eqref{eq:Gibbsmin} are equivalent to \eqref{eq:speciation}, and that a solution exists and is unique under appropriate assumptions \cite{Jonval2023,Jonval2024}.

To alleviate the nonlinearities of \eqref{eq:potential} with respect to the unknown $\boldsymbol{\textsf{n}}$, Jonval et al. \cite{JonvalBGCFT2025} reformulated system \eqref{eq:speciation} using $(\boldsymbol{\textsf{x}},\omega)\in\mathbb{R}^{|\mathcal{C}|}\times\mathbb{R}$ as unknowns, with
\[
\textsf{x}_i = \frac{\textsf{n}_i}{\textsf{n}_1 + \textsf{n}_2 + \ldots + \textsf{n}_{|\mathcal{C}|} }, \qquad\quad \omega = \frac{1}{\textsf{n}_1 + \textsf{n}_2 + \ldots + \textsf{n}_{|\mathcal{C}|}}.
\]
These are respectively the molar fractions of the species and the reciprocal of the total number of moles. It is straightforward to show that \eqref{eq:speciation}--\eqref{eq:potential} is equivalent to the enlarged system
\begin{subequations}\label{eq:enlarged}
    \begin{alignat}{3}
    \textsf{A} & \boldsymbol{\textsf{x}}  & & - \omega \textsf{f} & & = \textsf{0},\label{eq:linear1}\\
    \boldsymbol{\textsf{1}}^{\!\top} & \boldsymbol{\textsf{x}} & & - 1 & & = 0,\label{eq:linear2}\\
    \textsf{S}^{\!\top} & \ln \boldsymbol{\textsf{x}} & & - \textsf{t} & &= \textsf{0},
\end{alignat}
\end{subequations}
where
\begin{equation}\label{eq:vecTref}
    \textsf{t}_\ell = - \frac{1}{\textsf{R}\textsf{T}} \sum_{j = 1}^{|\mathcal{C}|} s_{\ell j} \mu^{\circ}_j,
\end{equation}
for $1\le \ell\le |\mathcal{R}|$. System \eqref{eq:enlarged} is now in the abstract form \eqref{eq:nnsystem}--\eqref{eq:nnsystemcondensed} with
\begin{equation}
    \nbsp = |\mathcal{C}| + 1, \qquad \nbel = |\mathcal{E}| + 1, \qquad \nbre = |\mathcal{R}|,
\end{equation}
and
\begin{equation}\label{eq:Axbt}
     A = \begin{bmatrix}
         \textsf{A} &  - \textsf{f}\\
         \boldsymbol{\textsf{1}}^{\!\top} & 0
     \end{bmatrix},
     \qquad
     \boldsymbol{x} = \begin{bmatrix}\,\boldsymbol{\textsf{x}} \,\\ \,\omega \,\end{bmatrix}, \qquad
     b = \begin{bmatrix} \, \textsf{0} \, \\ \,1\, \end{bmatrix},
     \qquad
     S = \begin{bmatrix} \textsf{S}^{\!\top} & \textsf{0} \end{bmatrix},
     \qquad
     t = \textsf{t}.
\end{equation}

For some simple models, it may happen that a suitable combination of the affine equations \eqref{eq:linear1}--\eqref{eq:linear2} enables us to isolate $\omega$ in a single affine equation. This equation is then decoupled from the remaining ones and will be used only to deduce $\omega$ at the end, once $\boldsymbol{\textsf{x}}$ is known. The system of equations satisfied by $\boldsymbol{\textsf{x}}$ is also in the abstract form \eqref{eq:nnsystem}, with
\begin{equation}
    \nbsp = |\mathcal{C}|, \qquad \nbel = |\mathcal{E}|, \qquad \nbre = |\mathcal{R}|.
\end{equation}
Such a situation arises, e.g., in the water dissociation problem presented in \S\ref{sse:descript.water}.

\subsection{Flatness matrix for a separated form}\label{sse:SeparatedForm}
By the column-wise formula \eqref{eq:flatMATjth} of Proposition~\ref{pr:flatMATexp}, the $j$-th column of $\mathcal{J}_H$ is given by $\nabla H^{-1} \partial_j [\nabla H ] \nabla H^{-1} H$. Because of \eqref{eq:separated}, for each $j=1, 2, \dots, \nbsp$, we have
\begin{equation}
    \partial_j [\nabla H ] = \begin{bmatrix}
    \, 0 & \dots & \partial_{jj} h_1 & \dots & 0 \,\\
    \, 0 & \dots & \partial_{jj} h_2 & \dots & 0 \,\\
    \, \vdots & \ddots & \vdots & \ddots & \vdots \,\\
    \, 0 & \dots & \partial_{jj} h_{\nbsp} & \dots & 0 \,
    \end{bmatrix}\!\!.
\end{equation}
Multiplying the above matrices by the vector $\nabla H^{-1} H =: (\eta_1,\eta_2,\dots,\eta_{\nbsp})^{\!\top} = \boldsymbol{\eta}$ and assembling the columns, we end up with
\begin{align}
    \mathcal{J}_H & = \nabla H^{-1} \begin{bmatrix}
        \, \partial_{11} h_1 \eta_1 & \partial_{22} h_1 \eta_2 &\dots&  \partial_{\nbsp\nbsp} h_1 \eta_{\nbsp} \,\\
        \, \partial_{11} h_2 \eta_1 & \partial_{22} h_2 \eta_2 &\dots&  \partial_{\nbsp\nbsp} h_2 \eta_{\nbsp} \,\\
        \, \vdots & \vdots &\ddots&  \vdots \,\\
        \, \partial_{11} h_{\nbsp} \eta_1 & \partial_{22} h_{\nbsp} \eta_2 &\dots& \partial_{\nbsp\nbsp} h_{\nbsp} \eta_{\nbsp} \,
    \end{bmatrix}\nonumber\\
    & = \nabla H^{-1} \begin{bmatrix}
        \, \partial_{11} h_1 & \partial_{22} h_1 &\dots&  \partial_{\nbsp\nbsp} h_1 \,\\
        \, \partial_{11} h_2 & \partial_{22} h_2 &\dots&  \partial_{\nbsp\nbsp} h_2 \,\\
        \, \vdots & \vdots &\ddots&  \vdots \,\\
        \, \partial_{11} h_{\nbsp} & \partial_{22} h_{\nbsp} &\dots& \partial_{\nbsp\nbsp} h_{\nbsp} \,
    \end{bmatrix}
    \begin{bmatrix}
        \, \eta_1 & 0 &\dots&  0 \,\\
        \, 0 & \eta_2 &\dots&  0 \,\\
        \, \vdots & \vdots &\ddots&  \vdots \,\\
        \, 0 & 0 &\dots& \eta_{\nbsp} \,
    \end{bmatrix}\nonumber\\[3pt]
    & = \nabla H^{-1} \mathfrak{D}^2 H \operatorname{Diag}(\nabla H^{-1} H),
\end{align}
as claimed by Lemma~\ref{le:separatedform}.

\subsection{Flatness matrix for the chemical system}\label{sse:flatchemgen}
Since $X_i = \exp Y_i$, we have $X'_i = x_i Y'_i$. Hence,
\begin{align}
    \nabla H =
    \begin{bmatrix}
    \, a_{11} x_1Y_1^{\prime} & a_{12} x_2Y_2^{\prime} & \dots & a_{1\nbsp} x_{\nbsp}Y_{\nbsp}^{\prime}\; \\
    \, a_{21} x_1Y_1^{\prime} & a_{22} x_2Y_2^{\prime} & \dots & a_{2\nbsp} x_{\nbsp}Y_{\nbsp}^{\prime}\; \\
    \, \vdots & \vdots & \ddots & \vdots\; \\
    \, a_{\nbel 1} x_1Y_1^{\prime} & a_{\nbel 2} x_2Y_2^{\prime} & \dots & a_{\nbel \nbsp} x_{\nbsp}Y_{\nbsp}^{\prime}\; \\
    \, s_{11}Y_1^{\prime} & s_{12}Y_2^{\prime} & \dots & s_{1\nbsp}Y_{\nbsp}^{\prime}\; \\
    \, s_{21}Y_1^{\prime} & s_{22}Y_2^{\prime} & \dots & s_{2\nbsp}Y_{\nbsp}^{\prime}\; \\
    \, \vdots & \vdots & \ddots & \vdots\; \\
    \, s_{\nbre 1}Y_1^{\prime} & s_{\nbre 2}Y_2^{\prime} & \dots & s_{\nbre \nbsp}Y_{\nbsp}^{\prime} \;
    \end{bmatrix} = \nabla G \, \operatorname{Diag}(\mathscr{Y}')
\end{align}
with $\mathscr{Y}'= (Y_1^{\prime}, Y_2^{\prime},\dots, Y_{\nbsp}^{\prime})^{\!\top}$. It follows that
\begin{equation}\label{eq:nablaH}
    \nabla H^{-1} = \operatorname{Diag}(\mathscr{Y}')^{-1} \nabla G^{-1}
\end{equation}
and
\begin{equation}
    \nabla H^{-1} H = \operatorname{Diag}(\mathscr{Y}')^{-1} \nabla G^{-1} H = \operatorname{Diag}(\mathscr{Y}')^{-1} \NLresnv,
\end{equation}
where $\NLresnv$ is the vector of normalized residuals defined in \eqref{eq:trueres}. Applying the $\operatorname{Diag}$ operator defined in \S\ref{sse:chem.model}, we get
\begin{equation}\label{eq:matresidual}
    \operatorname{Diag} (\nabla H^{-1} H) =  \operatorname{Diag}(\mathscr{Y}')^{-1} \operatorname{Diag}(\NLresnv).
\end{equation}

We now strive to compute $\mathfrak{D}^2 H$, defined in Lemma~\ref{le:separatedform}. Differentiating  $X'_i = x_i Y'_i$, we get $X''_i = x_i (Y''_i + (Y'_i)^2)$. Therefore,
\begin{align}
\mathfrak{D}^2 H & =\begin{bmatrix}
    \, a_{11} X_1^{\prime\prime} & a_{12} X_2^{\prime\prime} & \dots & a_{1\nbsp} X_{\nbsp}^{\prime\prime}\; \\
    \, a_{21} X_1^{\prime\prime} & a_{22} X_2^{\prime\prime} & \dots & a_{2\nbsp} X_{\nbsp}^{\prime\prime}\; \\
    \, \vdots & \vdots & \ddots & \vdots\; \\
    \, a_{\nbel 1} X_1^{\prime\prime} & a_{\nbel 2} X_2^{\prime\prime} & \dots & a_{\nbel \nbsp} X_{\nbsp}^{\prime\prime}\; \\
    \, s_{11}Y_1^{\prime\prime} & s_{12}Y_2^{\prime\prime} & \dots & s_{1\nbsp}Y_{\nbsp}^{\prime\prime}\; \\
    \, s_{21}Y_1^{\prime\prime} & s_{22}Y_2^{\prime\prime} & \dots & s_{2\nbsp}Y_{\nbsp}^{\prime\prime}\; \\
    \, \vdots & \vdots & \ddots & \vdots\; \\
    \, s_{\nbre 1}Y_1^{\prime\prime} & s_{\nbre 2}Y_2^{\prime\prime} & \dots & s_{\nbre \nbsp}Y_{\nbsp}^{\prime\prime}\;\,
    \end{bmatrix}\nonumber\\
    & =\begin{bmatrix}
    \, a_{11} x_1 (Y_1^{\prime\prime} + (Y_1^{\prime})^2) & a_{12} x_2 (Y_2^{\prime\prime} + (Y_2^{\prime})^2) & \dots & a_{1\nbsp} x_{\nbsp} (Y_{\nbsp}^{\prime\prime} + (Y_{\nbsp}^{\prime})^2)\; \\
    \, a_{21} x_1 (Y_1^{\prime\prime} + (Y_1^{\prime})^2) & a_{22} x_2 (Y_2^{\prime\prime} + (Y_2^{\prime})^2) & \dots & a_{2\nbsp} x_{\nbsp} (Y_{\nbsp}^{\prime\prime} + (Y_{\nbsp}^{\prime})^2)\; \\
    \, \vdots & \vdots & \ddots & \vdots\; \\
    \, a_{\nbel 1} x_1 (Y_1^{\prime\prime} + (Y_1^{\prime})^2) & a_{\nbel 2} x_2 (Y_2^{\prime\prime} + (Y_2^{\prime})^2) & \dots & a_{\nbel \nbsp} x_{\nbsp} (Y_{\nbsp}^{\prime\prime} + (Y_{\nbsp}^{\prime})^2)\; \\
    \, s_{11}Y_1^{\prime\prime} & s_{12}Y_2^{\prime\prime} & \dots & s_{1\nbsp}Y_{\nbsp}^{\prime\prime}\; \\
    \, s_{21}Y_1^{\prime\prime} & s_{22}Y_2^{\prime\prime} & \dots & s_{2\nbsp}Y_{\nbsp}^{\prime\prime}\; \\
    \, \vdots & \vdots & \ddots & \vdots\; \\
    \, s_{\nbre 1}Y_1^{\prime\prime} & s_{\nbre 2}Y_2^{\prime\prime} & \dots & s_{\nbre \nbsp}Y_{\nbsp}^{\prime\prime}\;\,
    \end{bmatrix}\nonumber\\
    &  =\nabla G\begin{bmatrix}
    \,Y_1^{\prime\prime} & 0 & \dots & 0\;\\
    \,0 & Y_2^{\prime\prime} & \dots & 0\;\\
    \, \vdots & \vdots & \ddots & \vdots\;\\
    \, 0 & 0 & \dots & Y_{\nbsp}^{\prime\prime}\;
    \end{bmatrix}
    +\bigg(I-\sum_{\ell=1}^{\nbre}\boldsymbol{e}_{\nbel+\ell} \boldsymbol{e}_{\nbel+\ell}^{\!\top}\bigg)\nabla G\begin{bmatrix}
    \,(Y_1^{\prime})^2 & 0 & \dots & 0\;\\
    \,0 & (Y_2^{\prime})^2 & \dots & 0\;\\
    \, \vdots & \vdots & \ddots & \vdots\;\\
    \, 0 & 0 & \dots & (Y_{\nbsp}^{\prime})^2\;\\
    \end{bmatrix}\nonumber\\[3pt]
    &  =\nabla G\operatorname{Diag}(\mathscr{Y}^{\prime\prime})
    + \bigg(I-\sum_{\ell=1}^{\nbre}\boldsymbol{e}_{\nbel+\ell} \boldsymbol{e}_{\nbel+\ell}^{\!\top}\bigg) \nabla G \operatorname{Diag}(\mathscr{Y}')^2, \label{eq:DfrakH}
\end{align}
where $\boldsymbol{e}_j$ is the $j$-th canonical basis vector and $\mathscr{Y}^{\prime\prime}= (Y_1^{\prime\prime}, Y_2^{\prime\prime}, \dots, Y_{\nbsp}^{\prime\prime})^{\!\top}$. Inserting \eqref{eq:nablaH}, \eqref{eq:matresidual} and \eqref{eq:DfrakH} into formula \eqref{eq:separatedflatmat} of Lemma~\ref{le:separatedform}, we end up with
\begin{align}
    \mathcal{J}_H &= \operatorname{Diag}(\mathscr{Y}')^{-1} \nabla G^{-1}
    \bigg[\nabla G \operatorname{Diag}(\mathscr{Y}^{\prime\prime}) \nonumber\\
       & \qquad\qquad\qquad\qquad\quad + \bigg(I-\sum_{\ell=1}^{\nbre}\boldsymbol{e}_{\nbel+\ell} \boldsymbol{e}_{\nbel+\ell}^{\!\top}\bigg)\nabla G \operatorname{Diag}(\mathscr{Y}')^2 \bigg]
     \operatorname{Diag}(\mathscr{Y}')^{-1} \operatorname{Diag}(\NLresnv) \nonumber\\
     & = \bigg\{\! \operatorname{Diag}(\mathscr{Y}')^{-2} \operatorname{Diag}(\mathscr{Y}^{\prime\prime}) + I 
     -\sum_{\ell=1}^{\nbre}
     \operatorname{Diag}(\mathscr{Y}')^{-1} \nabla G^{-1} \boldsymbol{e}_{\nbel+\ell} \boldsymbol{e}_{\nbel+\ell}^{\!\top} \nabla G \operatorname{Diag}(\mathscr{Y}')
     \bigg\} \operatorname{Diag}(\NLresnv). \label{eq:bigchunk}
\end{align}
On the other hand, $\nabla G^{-1}\boldsymbol{e}_{\nbel+\ell} = (r_{\ell 1},\dots,r_{\ell\nbsp})^{\!\top}$ by the definition \eqref{eq:inversenablaG} of
$r_{\ell i}$, while reading off the $(\nbel+\ell)$-th row of \eqref{eq:nablaG} gives
$\boldsymbol{e}_{\nbel+\ell}^{\!\top} \nabla G = [s_{\ell 1}, s_{\ell 2}, \dots, s_{\ell \nbsp}]$.
As a consequence,
\begin{equation}\setlength{\arraycolsep}{2.5pt}
    \operatorname{Diag}(\mathscr{Y}')^{-1} \nabla G^{-1} \boldsymbol{e}_{\nbel+\ell} \boldsymbol{e}_{\nbel+\ell}^{\!\top} \nabla G \operatorname{Diag}(\mathscr{Y}') = \begin{bmatrix}
    \, r_{\ell 1} s_{\ell 1}  & r_{\ell 1} s_{\ell 2} \dfrac{Y_2^{\prime}}{Y_1^{\prime}} & \dots & r_{\ell 1} s_{\ell \nbsp} \dfrac{Y_{\nbsp}^{\prime}}{Y_1^{\prime}} \;\, \\
    \, r_{\ell 2} s_{\ell 1} \dfrac{Y_1^{\prime}}{Y_2^{\prime}} & r_{\ell 2} s_{\ell 2} & \dots & r_{\ell 2} s_{\ell \nbsp} \dfrac{Y_{\nbsp}^{\prime}}{Y_2^{\prime}} \;\, \\
    \, \vdots & \vdots & \ddots & \vdots \;\, \\
    \, r_{\ell\nbsp} s_{\ell 1} \dfrac{Y_1^{\prime}}{Y_{\nbsp}^{\prime}} & r_{\ell\nbsp} s_{\ell 2} \dfrac{Y_2^{\prime}}{Y_{\nbsp}^{\prime}} & \dots & r_{\ell\nbsp} s_{\ell\nbsp} \;\,
    \end{bmatrix}\!\!,
\end{equation}
for all $\ell\in \{1,2,\dots,\nbre\}$. Summing over $\ell$, the last sum in \eqref{eq:bigchunk} equals $\Lambda^{-1}\nabla G^{-1}\Pi\nabla G\Lambda$ with $\Lambda=\operatorname{Diag}(\mathscr{Y}')$, and the same two identities give $(\nabla G^{-1}\Pi\nabla G)_{ij} = \sum_{\ell} r_{\ell i}s_{\ell j}$. Moreover, $\operatorname{Diag}(\mathscr{Y}')^{-2} \operatorname{Diag}(\mathscr{Y}^{\prime\prime}) + I = \operatorname{Diag}(\boldsymbol{\aleph})$ by \eqref{eq:alephtaui}. Since diagonal matrices commute, \eqref{eq:bigchunk} becomes $\mathcal{J}_H = \Lambda^{-1}\big[\operatorname{Diag}(\boldsymbol{\aleph}) - \nabla G^{-1}\Pi\nabla G\big]\Lambda\operatorname{Diag}(\NLresnv)$, which is \eqref{flat.mat00}. This completes the proof of Theorem~\ref{th:flatmat00}.

\subsection{Minimizing the upper bound of the flatness number}\label{sse:minupperflatchem}
For the CD parametrizations $\fs_i \mapsto (\textsf{X}_{\alpha_i}(\fs_i), \textsf{Y}_{\alpha_i}(\fs_i))$ defined in \eqref{eq:CD33para}, we have
\begin{equation}\label{CD_Yij}
    \aleph_i = 1+\dfrac{Y_i^{\prime\prime}}{\left(Y_i^{\prime}\right)^2}=\alpha_i,
\end{equation}
which is plainly due to the fact that
\begin{equation}
    Y'_i(\fs_i) = \frac{1}{1+ (1-\alpha_i)\fs_i}, \qquad Y''_i(\fs_i) = - \frac{1-\alpha_i}{[1+ (1-\alpha_i)\fs_i]^2}.
\end{equation}
Thus, instead of explicitly depending on the $\fs_i$'s, the quantities $\aleph_i$ in \eqref{eq:Pm1JhP} now become tuning weights $\alpha_i$ at our disposal. Plugging \eqref{CD_Yij} into \eqref{eq:Pm1JhP}, we get
\begin{equation}\label{flat.equiv}
    \Lambda\,\mathcal{J}_{H_{\boldsymbol{\alpha}}}\,\Lambda^{-1}
    = \big[\operatorname{Diag}(\boldsymbol{\alpha}) - \nabla G^{-1}\Pi\,\nabla G\big]
      \operatorname{Diag}(\NLresnv),
\end{equation}
where $\Lambda\,\mathcal{J}_{H_{\boldsymbol{\alpha}}}\,\Lambda^{-1} = P^{-1}\mathcal{J}_{H_{\boldsymbol{\alpha}}} P$ with $P = \operatorname{Diag}(\mathscr{Y}')^{-1} = \operatorname{Diag}( Y_1^{\prime},Y_2^{\prime},\dots,Y_{\nbsp}^{\prime})^{-1}$.

By the entrywise formula of Theorem~\ref{th:flatmat00}, the upper bound \eqref{eq:flatupperbound} of the flatness number can then be computed as
\begin{equation}
    (\widehat{\rho}_{H_{\boldsymbol{\alpha}}})^2 = \sum_{i=1}^{\nbsp} \bigg(\alpha_i - \displaystyle\sum_{\ell=1}^{\nbre}r_{\ell i} s_{\ell i} \bigg)^{\!2}\NLresns_i^2 + \diamond.
\end{equation}
The missing terms, represented by ``$\diamond$'' in the above equation, namely
\[
    \diamond=\sum_{j=1}^{\nbsp}\sum_{i\neq j}\bigg(\sum_{\ell=1}^{\nbre} r_{\ell i}s_{\ell j}\bigg)^{\!2}\NLresns_j^2,
\]
do not depend on $\boldsymbol{\alpha}$ at a fixed value of $\boldsymbol{x}$. Therefore, minimizing $(\widehat{\rho}_{H_{\boldsymbol{\alpha}}})^2$ with respect to $\boldsymbol{\alpha}\in [0,1]^{\nbsp}$ is a sum of squares problem in separated variables. A minimizer, unique whenever all the $\NLresns_i$'s are nonzero, is
\begin{equation}
     \alpha_i= \operatorname{Proj}_{[0,1]} \bigg( \sum_{\ell=1}^{\nbre}r_{\ell i} s_{\ell i} \bigg)
\end{equation}
for $i\in \{1,2,\ldots,\nbsp\}$. This establishes Proposition~\ref{prop:CDoptimalpha}.

Here, we have used a well-known fact in optimization, namely that a convex function of one variable can be minimized on an interval of $\mathbb{R}$ by first minimizing it over $\mathbb{R}$ and then projecting the minimizer onto the interval \cite{Ciarlet1988}.

\subsection{Simultaneous vanishing of the two normalized residuals}\label{proof:simultvanishRes}
Throughout this proof, we write $\nabla G(\fv)$ for $\nabla G(\mathscr{Y}(\fv))$, and $\boldsymbol{y}^k=\mathscr{Y}(\fv_g,\fv_b)$, where $(\fv_g,\fv_b)=\fv^k$. For a matrix $M$ with $\nbsp$ rows, $[M]_b$ denotes the block of its rows indexed by $\mathcal{B}$, and for a vector $\boldsymbol{v}$, $\{\boldsymbol{v}\}_b$ denotes its subvector indexed by $\mathcal{B}$. Up to a permutation of the unknowns, we may assume that the bad unknowns come last, and we set $\nbbad:=|\mathcal{B}|$.
We now verify that the system of equations consisting of the bad nonlinear normalized residuals
\begin{equation}\label{eq:myDeltab}
    \NLresnv_b(\fv_g,\widetilde{\fv}_b) = \big\{ \nabla G(\fv_g,\widetilde{\fv}_b)^{-1} H(\fv_g,\widetilde{\fv}_b)\big\}_b = 0
\end{equation}
in the unknown $\widetilde{\fv}_b$ is equivalent to that made up of the bad linear normalized residuals
\begin{equation}\label{eq:myXib}
    \resnv_b(\boldsymbol{y}^{k};\fv_g, \widetilde{\fv}_b)=  \big\{ \nabla G(\fv_g,\fv_b)^{-1} H(\fv_g,\widetilde{\fv}_b) \big\}_b = 0.
\end{equation}
Split the columns of the log-trick Jacobian \eqref{eq:nablaG} according to the good--bad
partition and write
\begin{equation}\label{apd:nablaG}
    \nabla G(\fv_g,\fv_b) = \big[\,\nabla_{\!g} G \mid \nabla_{\!b} G(\fv_b)\,\big],
\end{equation}
where $\nabla_{\!g} G\in\mathbb{R}^{\nbsp\times (\nbsp-\nbbad)}$ gathers the columns indexed by the
good components and $\nabla_{\!b} G$ the remaining $\nbbad$ ones. Only the latter depend on
$\fv_b$, since, by \eqref{eq:nablaG}, the $j$-th column of $\nabla G$ involves $\boldsymbol{x}$
through its $j$-th entry alone. Consequently, the two matrices
$\nabla G(\fv_g,\fv_b)$ and $\nabla G(\fv_g,\widetilde{\fv}_b)$ share the same block
$\nabla_{\!g} G$, and differ only in their last $\nbbad$ columns.

Set $\mathcal{W} := \operatorname{ran}\nabla_{\!g} G$. Both Jacobians being invertible, $\nabla_{\!g} G$ has full
column rank, so that $\dim \mathcal{W} = \nbsp-\nbbad$ and $\dim \mathcal{W}^{\perp} = \nbbad$. Now let $M$ denote either
$\nabla G(\fv_g,\fv_b)$ or $\nabla G(\fv_g,\widetilde{\fv}_b)$. Since $M^{-1}M = I$, the last
$\nbbad$ rows of $M^{-1}$ annihilate every column of $\nabla_{\!g} G$, hence belong (once transposed) to $\mathcal{W}^{\perp}$.
Being $\nbbad$ rows of an invertible matrix, they are linearly independent, and they therefore form
a basis of $\mathcal{W}^{\perp}$.

Applying this to both choices of $M$, the row blocks $[\nabla G(\fv_g,\fv_b)^{-1}]_b$ and
$[\nabla G(\fv_g,\widetilde{\fv}_b)^{-1}]_b$ are two bases of the same $\nbbad$-dimensional
space $\mathcal{W}^{\perp}$. Hence there exists an invertible $\nbbad\times \nbbad$ change-of-basis
matrix $Q$ such that
\begin{align}\label{eq:2basiseq}
    \big[ \nabla G(\fv_g,\fv_b)^{-1}\big]_b = Q \big[\nabla G(\fv_g,\widetilde{\fv}_b)^{-1}\big]_b.
\end{align}

Equation \eqref{eq:2basiseq} means that by a linear combination of the bad rows of $\nabla G(\fv_g,\widetilde{\fv}_b)^{-1}$, we can recover any bad row of $\nabla G(\fv_g,\fv_b)^{-1}$.
Right-multiplication by $H(\fv_g,\widetilde{\fv}_b)$ brings equality \eqref{eq:2basiseq} to
\begin{equation}
    \big[ \nabla G(\fv_g,\fv_b)^{-1}\big]_b  H(\fv_g,\widetilde{\fv}_b) = Q \big[ \nabla G(\fv_g,\widetilde{\fv}_b)^{-1}\big]_b  H(\fv_g,\widetilde{\fv}_b),
\end{equation}
which is the same as
\begin{equation}\label{eq:passage}
    \big\{ \nabla G(\fv_g,\fv_b)^{-1} H(\fv_g,\widetilde{\fv}_b)\big\}_b
    = Q \big\{\nabla G(\fv_g,\widetilde{\fv}_b)^{-1} H(\fv_g,\widetilde{\fv}_b)\big\}_b.
\end{equation}
Equivalence between \eqref{eq:myDeltab} and \eqref{eq:myXib} results from \eqref{eq:passage} and from the invertibility of $Q$.
Theorem~\ref{th:equivNLresiduals} is thus verified.

\subsection{Bad subset for AANC}\label{sse:badAANC}
Since multiplying the objective function by a positive constant does not change the location of the maximizer, we have
\begin{equation}
    j_* = \operatorname{arg} \max_{0\le j\leq\nbsp} \Big\{\dfrac{\nbsp}{\widecheck{\rho}^2_H}R_j-j \Big\}
    = \operatorname{arg} \max_{0\le j\leq\nbsp} \Big\{R_j-j\dfrac{\widecheck{\rho}^2_H}{\nbsp} \Big\}.
\end{equation}
From the definition \eqref{eq:cumulativeSum} of $R_j$, it follows that, for $0\le j\le\nbsp$,
\begin{equation}\label{eq:newobjectif}
    R_j-j\dfrac{\widecheck{\rho}^2_H}{\nbsp} = \sum_{\ell=1}^j \Big(\Gamma_{(\ell)}-\dfrac{\widecheck{\rho}^2_H}{\nbsp}\Big).
\end{equation}
We note that
\[
R_1 = \Gamma_{(1)} \ge \frac{\widecheck{\rho}^2_H}{\nbsp}
\]
because $\Gamma_{(1)}$ is the greatest element of the family $(\Gamma_{(1)}, \ldots, \Gamma_{(\nbsp)})$ and $\widecheck{\rho}^2_H/\nbsp$ is their average. Therefore, the first summand in the right-hand side of \eqref{eq:newobjectif} is nonnegative, that is,
\begin{equation}
    \Gamma_{(1)} - \frac{\widecheck{\rho}^2_H}{\nbsp} \ge 0.
\end{equation}
Since the sequence $(\Gamma_{(\ell)})_{1\le \ell\le \nbsp}$ is nonincreasing, we also have
\begin{equation}
    \Gamma_{(1)} - \frac{\widecheck{\rho}^2_H}{\nbsp} \ge \Gamma_{(2)} - \frac{\widecheck{\rho}^2_H}{\nbsp} \ge \ldots \ge \Gamma_{(\nbsp)} - \frac{\widecheck{\rho}^2_H}{\nbsp},
\end{equation}
where the last term is nonpositive, i.e.,
\begin{equation}
    \Gamma_{(\nbsp)} - \frac{\widecheck{\rho}^2_H}{\nbsp} \le 0.
\end{equation}
Thus, the sequence $( \Gamma_{(\ell)} - \widecheck{\rho}^2_H/\nbsp )_{1\le \ell\le\nbsp}$ is nonincreasing and goes from nonnegative to nonpositive values.
Consequently, the partial sums \eqref{eq:newobjectif} increase strictly with $j$ as long as $\Gamma_{(j)} > \widecheck{\rho}^2_H/\nbsp$, and they do not increase afterwards. Their maximum over $\{0,1,\ldots,\nbsp\}$ is therefore attained at the index $j_*$ given by \eqref{eq:solmaxecart}, which is the smallest maximizer. If all the $\Gamma_j$'s are equal, all the partial sums vanish, the set in \eqref{eq:solmaxecart} is empty, and the convention $\max\emptyset=0$ gives $j_*=0$, which is again the smallest maximizer. This proves Proposition~\ref{pr:solminjnc}.

\subsection{Description of the chemical systems}\label{sse:testcases}
\subsubsection{H${}_2$O test case}\label{sse:descript.water}
The problem of water dissociation consists of 3 species, 2 elements and 1 reaction summarized by
\begin{subequations}
\begin{align}
    \mathcal{C} &= (
\mathrm{H_2O},\, \mathrm{H^+},\,\mathrm{OH^-}),\\
\mathcal{E} &= (\mathrm{H},\, \mathrm{O}),\\
\mathcal{R} &= (\Hp + \OHm  \rightleftharpoons \HHO).
\end{align}
\end{subequations}
The corresponding atomic and stoichiometric matrices are
\begin{equation}
    \textsf{A}=\begin{bmatrix}
        2 & 1 & 1 \\
        1 & 0 & 1
    \end{bmatrix}\!\!, \qquad \quad\textsf{S}^{\!\top} = \begin{bmatrix}
        1  & -1 & -1
    \end{bmatrix}\!\!.
\end{equation}
Given the feeds $(\textsf{f}_{\mathrm{H}}, \textsf{f}_{\mathrm{O}})$, system \eqref{eq:enlarged} reads
\begin{subequations}\label{eq:systmolfrac}
\begin{alignat}{5}
    2 x_1 \ &+ \  & x_2 \ &+ & \ x_3 \ &- & \, \textsf{f}_{\mathrm{H}}  \omega  \ && &=\  0, \label{eq:firsteq}\\
     x_1 \ &    &  \ &+ \   & x_3 \ &- & \, \textsf{f}_{\mathrm{O}}  \omega \ && &=\ 0, \label{eq:sectarget}\\
     x_1 \ &+ \ & x_2   \ & +   & \ x_3 \ && &-\ &1 &=\   0, \label{eq:thirdeq}\\
    \ln x_1 \ &- \ & \ln x_2  &- \ &  \ln x_3 \ && &-\ &t &=\ 0,
\end{alignat}
\end{subequations}
with
\begin{equation}
    t=-\frac{1}{\textsf{R}\textsf{T}}(\mu_{\mathrm{H}_2 \mathrm{O}}^{\circ}-\mu_{\mathrm{H}^{+}}^{\circ}-\mu_{\mathrm{OH}^{-}}^{\circ}).
\end{equation}

Instead of solving the $4\times 4$ system \eqref{eq:systmolfrac}, we proceed with the following transformation so as to highlight a dimensionless parameter $\nu$ representing the charge imbalance between the ions $\mathrm{H}^{+}$ and $\mathrm{OH}^{-}$ and directly affecting the desired molar fractions $(x_1,x_2,x_3)$.
Replacing the second equation \eqref{eq:sectarget} by the linear combination
\[
\eqref{eq:sectarget} - \frac{\textsf{f}_{\mathrm{O}}}{\textsf{f}_{\mathrm{H}}}\times\eqref{eq:firsteq} - \Big( 1 - \frac{2\textsf{f}_{\mathrm{O}}}{\textsf{f}_{\mathrm{H}}}\Big)\times\eqref{eq:thirdeq},
\]
multiplying the result by $-2$, and making use of the dimensionless number
\begin{equation}\label{eq:nu.def}
    \nu=1-2\dfrac{\textsf{f}_{\mathrm{O}}}{\textsf{f}_{\mathrm{H}}},
\end{equation}
we end up with the system
\begin{subequations}\label{eq:44system}
\begin{alignat}{4}
    2x_{1}  &\,+\,  &x_{2}\   &+\,  &x_{3}\     &-\ &\textsf{f}_{\mathrm{H}}\omega &=\ 0,\\
    &   &(1+\nu)\, x_{2}\ &-\, &\,(1-\nu)\, x_{3}\ &-\ &2\nu &=\ 0,\\
    x_{1}     &\,+    &x_{2}\   &+\,  &x_{3}\     &- &1\                        &=\ 0,\\
    \ln x_{1} &\,- &\ln x_{2}\  &-\,  &\ln x_{3}\   &- &t\                        &=\ 0,
\end{alignat}
\end{subequations}
in the 4 unknowns $(x_1,x_2,x_3,\omega)$. Since the unknown $\omega$ appears only in the first equation, which is moreover linear, it can be effortlessly deduced from $(x_1,x_2,x_3)$. We can thus focus on the $\nu$-dependent $3\times 3$ subsystem
\begin{subequations}\label{eq:watermodel33}
\begin{alignat}{5}
    f_1:=& & &   &(1+\nu)\, x_{2}\ &-\, &\,(1-\nu)\, x_{3}\ &-\ &2\nu &=\ 0,\\
    f_2:=& &x_{1}    &\,+    &x_{2}\   &+\,  &x_{3}\     &- &1\                        &=\ 0, \\
    f_3:=& &\ \ln x_{1} &\,- &\ln x_{2}\  &-\,  &\ln x_{3}\   &- &t\                        &=\ 0,
\end{alignat}
\end{subequations}
in the 3 unknowns  $(x_1,x_2,x_3)$.

Experiments are conducted for $\nu = 10^{-1}$ and $\nu = 10^{-7}$. The reference potentials \cite{JohnsonOelkersHelgeson1992} (see also \cite[Appendix A]{Jonval2024}), in $\mathrm{J\,mol^{-1}}$ at $\textsf{P} = 1$ bar and $\textsf{T} = 298.15\,\mathrm{K}$, are
\begin{gather*}
\mu^\circ_{\HHO} = - 237138.97589284607, \quad \mu^\circ_{\Hp} = 9956.885403312557, \\
\mu^\circ_{\OHm} = - 147340.5566329419,
\end{gather*}
so that $t = 40.2406556362$ with $\textsf{R} = 8.3145\,\mathrm{J\,mol^{-1}\,K^{-1}}$.

\subsubsection{NaCl test case}\label{sse:descript.NaCl}
We next consider a larger chemical system, hereafter referred to as $\NaCl$, composed of 8 species, 4 elements and 4 reactions. More specifically,
\begin{subequations}
\begin{align}
\mathcal{C} &= (
\mathrm{H_2O},\, \mathrm{Cl^-},\, \mathrm{Na^+},\, \mathrm{H^+},\,
\mathrm{NaCl},\, \mathrm{NaOH},\, \mathrm{OH^-},\, \mathrm{HCl}),\\
\mathcal{E} &= (\mathrm{H},\, \mathrm{O},\, \mathrm{Na},\, \mathrm{Cl}),\\
\mathcal{R} &= \begin{aligned}[t]
       (\mathrm{NaCl} &\rightleftharpoons \mathrm{Cl^-} + \mathrm{Na^+}, & 
        \mathrm{NaOH} &\rightleftharpoons \mathrm{H_2O} + \mathrm{Na^+} - \mathrm{H^+},\\[-0.75mm]
        \mathrm{OH^-} &\rightleftharpoons \mathrm{H_2O} - \mathrm{H^+}, &
        \mathrm{HCl} & \rightleftharpoons \mathrm{Cl^-} + \mathrm{H^+}).
\end{aligned}
\end{align}
\end{subequations}
The corresponding matrices are
\begin{equation}
    \textsf{A} = \begin{bmatrix}
2 & 0 & 0 & 1 & 0 & 1 & 1 & 1 \\
1 & 0 & 0 & 0 & 0 & 1 & 1 & 0 \\
0 & 0 & 1 & 0 & 1 & 1 & 0 & 0 \\
0 & 1 & 0 & 0 & 1 & 0 & 0 & 1
\end{bmatrix}\!\!, \quad
    \textsf{S}^{\!\top} = \begin{bmatrix}
0 & 1 & 1 & 0 & \!\!\!\!-1 & 0 & 0 & 0 \\
1 & 0 & 1 &  \!\!\!\!-1 & 0 &  \!\!\!\!-1 & 0 & 0 \\
1 & 0 & 0 &  \!\!\!\!-1 & 0 & 0 &  \!\!\!\!-1 & 0 \\
0 & 1 & 0 & 1 & 0 & 0 & 0 &  \!\!\!\!-1
\end{bmatrix}\!\!.
\end{equation}
The feeds are prescribed as
\begin{equation}
    \textsf{f}_{\mathrm{H}} = 111.0174, \qquad \textsf{f}_{\mathrm{O}} = 55.5087, \qquad
    \textsf{f}_{\mathrm{Na}} = 0.1, \qquad \textsf{f}_{\mathrm{Cl}} = 0.1,
\end{equation}
while the $\textsf{t}_\ell$'s defined in \eqref{eq:vecTref} are deduced from the reference chemical potentials \cite{JohnsonOelkersHelgeson1992,Jonval2024}
\begin{subequations}
    \begin{alignat}{2}
    \mu_{\HHO}^\circ &= - 237138.97589284607 & \qquad \quad \mu_{\mathrm{NaCl}}^\circ & = - 378778.4933733225 \\[-0.75mm]
    \mu_{\mathrm{Cl}^-}^\circ & = - 121332.84714937139 & \qquad \quad \mu_{\mathrm{NaOH}}^\circ & = - 408024.6192443839 \\[-0.75mm]
    \mu_{\mathrm{Na}^+}^\circ & = -251923.79553026147 & \qquad \quad \mu_{\OHm}^\circ & = - 147340.5566329419\\[-0.75mm]
    \mu_{\Hp}^\circ & = 9956.885403312557 & \qquad \quad \mu_{\mathrm{HCl}}^\circ & = - 117278.5194378308
    \end{alignat}
\end{subequations}
at the temperature $\textsf{T} = 298.15 \mathrm{K}$. This time, the enlarged system \eqref{eq:enlarged} is used as it stands.

\subsubsection{Seawater test case}\label{sse:descript.seawater}
Finally, we consider an even larger chemical system, hereafter referred to as Seawater, made up of 37 species, 10 elements and 27 reactions. More specifically,
\begin{subequations}
\begin{align}
    \mathcal{C} &= \begin{aligned}[t]
(
&\mathrm{H_2O},\, \mathrm{Na^+},\, \mathrm{Mg^{2+}},\, \mathrm{SO_4^{2-}},\, \mathrm{Ca^{2+}},\, \mathrm{K^+},\, \mathrm{HCO_3^-},\, \mathrm{Sr^{2+}},\, \mathrm{Cl^-},\, \mathrm{H^+},\\[-0.75mm]
&\mathrm{CaOH^+},\, \mathrm{CO_2},\, \mathrm{CO_3^{2-}},\, \mathrm{CaHCO_3^+}, \mathrm{CaCl^+},\, \mathrm{CaCl_2},\, \mathrm{CaSO_4},\, \mathrm{HSO_4^-},\\[-0.75mm]
&\mathrm{KCl}, \,\mathrm{KHSO_4},\, \mathrm{KOH},\, \mathrm{KSO_4^-},\, \mathrm{Mg(CO_3)},\, \mathrm{Mg(HCO_3)^+},\, \mathrm{MgCl^+}, \\[-0.75mm]
&\mathrm{MgOH^+},\, \mathrm{NaCl}, \,\mathrm{NaOH},\, \mathrm{OH^-},\, \mathrm{Sr(HCO_3)^+},\, \mathrm{SrCl^+},\, \mathrm{SrOH^+},\\[-0.75mm]
&\mathrm{NaSO_4^-},\, \mathrm{MgSO_4}, \, \mathrm{HCl},\, \mathrm{CaCO_3},\, \mathrm{SrCO_3}
),
\end{aligned}\\
\mathcal{E} &=
(\mathrm{H},\, \mathrm{O},\, \mathrm{Na},\, \mathrm{Mg},\, \mathrm{S},\, \mathrm{Ca},\, \mathrm{K},\, \mathrm{C},\, \mathrm{Sr},\, \mathrm{Cl}),\\
\mathcal{R} & = \text{cf. Table }\, \ref{ta:Rseawater}.
\end{align}
\end{subequations}
The feeds are prescribed as
\begin{subequations}
\begin{alignat}{5}
    \textsf{f}_{\mathrm{H}} & = 110.78054, & \, \textsf{f}_{\mathrm{O}} & = 55.5087, & \,
    \textsf{f}_{\mathrm{Na}} & = 0.469, & \, \textsf{f}_{\mathrm{Mg}} & = 0.0528, & \, \textsf{f}_{\mathrm{S}} & = 0.0282,\\[-0.75mm]
    \textsf{f}_{\mathrm{Ca}} & = 0.0103, & \, \textsf{f}_{\mathrm{K}} & = 0.0102, & \, \textsf{f}_{\mathrm{C}} & = 0.00206, & \, \textsf{f}_{\mathrm{Sr}} & = 10^{-5},
    & \, \textsf{f}_{\mathrm{Cl}} & = 0.546,
\end{alignat}
\end{subequations}
while the $\textsf{t}_k$'s defined in \eqref{eq:vecTref} stem from the reference chemical potentials of Table~\ref{ta:Mseawater}.
Here again, the enlarged system \eqref{eq:enlarged} is used as it stands. All initial guesses deployed in this model are built in the same way as for the $\NaCl$ test case.

\begin{table}[h]
\begin{alignat*}{2}
\mathrm{CaOH^+} &\rightleftharpoons \mathrm{H_2O} + \mathrm{Ca^{2+}} - \mathrm{H^+};
& \qquad\quad
\mathrm{CO_2} &\rightleftharpoons -\mathrm{H_2O} + \mathrm{HCO_3^-} + \mathrm{H^+}; \\[-0.75mm]
\mathrm{CO_3^{2-}} &\rightleftharpoons \mathrm{HCO_3^-} - \mathrm{H^+};
& \qquad\quad
\mathrm{CaHCO_3^+} &\rightleftharpoons \mathrm{Ca^{2+}} + \mathrm{HCO_3^-}; \\[-0.75mm]
\mathrm{CaCl^+} &\rightleftharpoons \mathrm{Ca^{2+}} + \mathrm{Cl^-};
& \qquad\quad
\mathrm{CaCl_2} &\rightleftharpoons \mathrm{Ca^{2+}} + 2\mathrm{Cl^-}; \\[-0.75mm]
\mathrm{CaSO_4} &\rightleftharpoons \mathrm{SO_4^{2-}} + \mathrm{Ca^{2+}};
& \qquad\quad
\mathrm{HSO_4^-} &\rightleftharpoons \mathrm{SO_4^{2-}} + \mathrm{H^+};\\[-0.75mm]
\mathrm{KCl} &\rightleftharpoons \mathrm{K^+} + \mathrm{Cl^-};
& \qquad\quad
\mathrm{KHSO_4} &\rightleftharpoons \mathrm{SO_4^{2-}} + \mathrm{K^+} + \mathrm{H^+}; \\[-0.75mm]
\mathrm{KOH} &\rightleftharpoons \mathrm{H_2O} + \mathrm{K^+} - \mathrm{H^+};
& \qquad\quad
\mathrm{KSO_4^-} &\rightleftharpoons \mathrm{SO_4^{2-}} + \mathrm{K^+}; \\[-0.75mm]
\mathrm{Mg(CO_3)} &\rightleftharpoons \mathrm{Mg^{2+}} + \mathrm{HCO_3^-} - \mathrm{H^+};
& \qquad\quad
\mathrm{Mg(HCO_3)^+} &\rightleftharpoons \mathrm{Mg^{2+}} + \mathrm{HCO_3^-}; \\[-0.75mm]
\mathrm{MgCl^+} &\rightleftharpoons \mathrm{Mg^{2+}} + \mathrm{Cl^-};
& \qquad\quad
\mathrm{MgOH^+} &\rightleftharpoons \mathrm{H_2O} + \mathrm{Mg^{2+}} - \mathrm{H^+};
\\[-0.75mm]
\mathrm{NaCl} &\rightleftharpoons \mathrm{Na^+} + \mathrm{Cl^-};
& \qquad\quad
\mathrm{NaOH} &\rightleftharpoons \mathrm{H_2O} + \mathrm{Na^+} - \mathrm{H^+}; \\[-0.75mm]
\mathrm{OH^-} &\rightleftharpoons \mathrm{H_2O} - \mathrm{H^+};
& \qquad\quad
\mathrm{Sr(HCO_3)^+} &\rightleftharpoons \mathrm{HCO_3^-} + \mathrm{Sr^{2+}}; \\[-0.75mm]
\mathrm{SrCl^+} &\rightleftharpoons \mathrm{Sr^{2+}} + \mathrm{Cl^-};
& \qquad\quad
\mathrm{SrOH^+} &\rightleftharpoons \mathrm{H_2O} + \mathrm{Sr^{2+}} - \mathrm{H^+}; \\[-0.75mm]
\mathrm{NaSO_4^-} &\rightleftharpoons \mathrm{Na^+} + \mathrm{SO_4^{2-}};
& \qquad\quad
\mathrm{MgSO_4} &\rightleftharpoons \mathrm{Mg^{2+}} + \mathrm{SO_4^{2-}}; \\[-0.75mm]
\mathrm{HCl} &\rightleftharpoons \mathrm{Cl^-} + \mathrm{H^+};
& \qquad\quad
\mathrm{CaCO_3} &\rightleftharpoons \mathrm{Ca^{2+}} + \mathrm{HCO_3^-} - \mathrm{H^+};\\[-0.75mm]
\mathrm{SrCO_3} &\rightleftharpoons \mathrm{HCO_3^-} + \mathrm{Sr^{2+}} - \mathrm{H^+}.
& \qquad\quad
\end{alignat*}
\caption{Set of 27 reactions in the Seawater system.}
\label{ta:Rseawater}
\end{table}
\begin{table}[h]
    \centering
    \begin{tabular}{l@{$\;\,$}l@{$\qquad\quad$}l@{$\;\,$}l}
        $\HHO$ & -237138.97589284607 &  $\mathrm{KHSO}_4$  & -1.0084285424401537e6 \\[0.25mm]
        $\mathrm{Na}^+$ & -251923.79553026147  & $\mathrm{KOH}$ & -427271.0297495857 \\[0.25mm]
        $\mathrm{Mg}^{2+}$  & -444027.90215494944  & $\mathrm{KSO}_4^-$  & -1.0219846680885215e6\\[0.25mm]
        $\mathrm{SO}_4^{2-}$ & -734502.084490013 & $\mathrm{Mg(CO_3)}$ & -989014.6934790071\\[0.25mm]
        $\mathrm{Ca}^{2+}$ & -542833.0422538111 & $\mathrm{Mg(HCO_3)^+}$ & -1.0368796788141541e6\\[0.25mm]
        $\mathrm{K}^+$ & -272504.8996797245 & $\mathrm{MgCl^+}$ & -574547.7790990678 \\[0.25mm]
        $\mathrm{HCO}_{3}^{-}$ & -576982.8987273659 & $\mathrm{MgOH^+}$ & -614525.8903919919\\[0.25mm]
        $\mathrm{Sr}^{2+}$ & -553878.8001745282 & $\mathrm{NaCl}$ & -378778.4933733225\\[0.25mm]
        $\mathrm{Cl}^-$ & -121332.84714937139 & $\mathrm{NaOH}$ & -408024.6192443839\\[0.25mm]
        $\mathrm{H}^+$ & 9956.885403312557 & $\mathrm{OH^-}$ & -147340.5566329419\\[0.25mm]
        $\mathrm{CaOH}^+$ & -706762.1524620559 & $\mathrm{Sr(HCO_3)^+}$ & -1.1478393365598267e6 \\[0.25mm]
        $\mathrm{CO}_2$ & -376017.06579503417 & $\mathrm{SrCl^+}$\,  & -683750.1584404652\\[0.25mm]
        $\mathrm{CO}_3^{2-}$ & -518026.1359207859 & $\mathrm{SrOH^+}$ & -715130.1518277868\\[0.25mm]
        $\mathrm{CaHCO}_3^+$ & -1.135747579476859e6 & $\mathrm{NaSO_4^-}$ & -1.0003785022780169e6 \\[0.25mm]
        $\mathrm{CaCl}^+$ & -672453.3605597073 & $\mathrm{MgSO_4}$ & -1.2012145947920694e6 \\[0.25mm]
        $\mathrm{CaCl}_2$ & -801738.9637317051 & $\mathrm{HCl}$ & -117278.5194378308\\[0.25mm]
        $\mathrm{CaSO}_4$ & -1.2993419278916481e6 & $\mathrm{CaCO_3}$ & -1.0898072308832686e6\\[0.25mm]
        $\mathrm{HSO}_4^-$ & -745798.9041292057 & $\mathrm{SrCO_3}$ & -1.0982170694298274e6\\[0.25mm]
        $\mathrm{KCl}$ & -389322.1898012777 &  &
    \end{tabular}
    \caption{Set of 37 reference chemical potentials $\mu^\circ$ for the Seawater system.}
    \label{ta:Mseawater}
\end{table}

\end{document}